\documentclass{amsart}
\usepackage{graphicx}
 \usepackage{stmaryrd}
\usepackage{color}
\def\D{\nabla}

\def\a{\alpha}
\def\b{\beta}

\def\({\left (}
\def\){\right )}
\def\<{\left\langle}
\def\>{\right\rangle}

 \newtheorem{thm}{Theorem}[section]

\newtheorem{lem}[thm]{Lemma}

\newtheorem*{rem}{Remark}

\newcommand{\abs}[1]{\left\vert#1\right\vert}

\newcommand{\R}{\mathbb R}

\newcommand{\pfrac}[2]{\frac{\partial #1}{\partial #2}}

\begin{document}
\title{The Yang-Mills flow in   vector bundles over higher dimensional Riemannian manifolds}

\numberwithin{equation}{section}

\author{Jared Casey}
\address{Department of Mathematics, The University of Queensland, Brisbane,
QLD 4072, Australia}
\email{jared.casey@uq.edu.au }

\author{Chak Hoi Chan }
\address{Department of Mathematics, The University of Queensland, Brisbane,
QLD 4072, Australia}
\email{c.h.chan@uq.edu.au}

 \author{Min-Chun Hong}
\address{Department of Mathematics, The University of Queensland, Brisbane,
QLD 4072, Australia}
\email{hong@maths.uq.edu.au}

\begin{abstract}We develop a parabolic formulation of the Uhlenbeck gauge fixing theorem and use it to prove the maximal existence of the Yang-Mills flow in   vector bundles over compact $n$-dimensional Riemannian manifolds for $n \geq 4$. We apply    Moser's iteration technique \cite{Mo1} to give a new proof of  the bound of the curvature. Using an idea of Hamilton on the Ricci flow \cite{Ha}, we also provide a new proof  on uniform estimates of higher order derivatives. Furthermore, we investigate the blow-up of the Yang-Mills flow at the maximal existence time and improve an asymptotical result of Hong-Tian \cite{HT1}.
\end{abstract}

\makeatletter
	\@namedef{subjclassname@2020}{%
	\textup{2020} Mathematics Subject Classification}
	\makeatother

 \maketitle

 \pagestyle{myheadings} \markright {The Yang-Mills Flow over Higher Dimensional Riemannian Manifolds}

\section{Introduction}
Let $M$ be a compact $n$-dimensional Riemannian manifold without boundary for $n \geq 4$, and $E$ a vector bundle  over $M$ with structure group $G$, a compact Lie group. Associated with each vector bundle $E$ is an adjoint bundle $\text{Ad} \ E$ whose fiber $(\text{Ad} \ E)_x \cong \mathfrak{g}$ is the Lie algebra of $G$. From the adjoint bundle $\text{Ad} \ E$ we can construct $\Omega^p(\text{Ad} \ E) := \Gamma(\text{Ad} \ E \otimes \Omega^p(M))$, the space of  of locally $\mathfrak{g}$-valued $p$-forms. For a connection $D_A$ in $E$, the Yang-Mills functional is given by
 \begin{equation}\label{YM}\mbox{YM} (A; M) :=  \int_M |F_A|^2 \,dv, \end{equation}
where $F_A = D_A \circ D_A \in \Omega^2(\text{Ad} \ E)$ is the curvature  and  $D_A = D_{\text{ref}} + A$ for some smooth reference connection $D_{\text{ref}}$ and $A \in \Omega^1(\text{Ad} \ E)$.
A connection $D_A$ is called Yang-Mills if it is a solution of the Yang-Mills equation
 \begin{equation}\label{1.1}D_A^*F_A =0\,   \end{equation}
which is the Euler-Lagrange equation of the Yang-Mills functional (\ref{YM}). The Yang-Mills flow is   defined as the corresponding evolution equation of connections $A(t)$ given by
\begin{equation}\label{YMF} \frac {\partial A}{\partial t}=-D_A^*F_A  \end{equation}
with   an initial connection $A(0) = A_0 \in \Omega^1(Ad E)$. For further details, we refer the reader to the seminal works of Uhlenbeck \cite{U2} and Struwe \cite{St3}.


The Yang-Mills flow (\ref{YMF}) was first considered by Atiyah-Bott \cite {AB} in their study of the Morse theory of the Yang-Mills functional over Riemann surfaces. It was then applied by Donaldson \cite {Do1} to establish the Donaldson-Uhlenbeck-Yau Theorem, which asserts the existence of unique Hermitian-Einstein connections in stable irreducible holomorphic vector bundles over K\"ahler surfaces (see also \cite {Do2}). This theorem was then generalized by Simpson \cite {Si} to holomorphic bundles over non-compact K\"ahler
manifolds (see an extension in \cite{Ho1}).  In unstable holomorphic vector bundles over  compact K\"ahler manifolds,  Bando-Siu \cite {BS}  conjectured  an equivalent relation between the limiting bundle of the Yang-Mills flow and the Harder-Narashimhan filtration on K\"ahler manifolds. This conjecture was settled for K\"ahler surfaces by Daskalopous-Wentworth \cite {DW}-\cite {DW1}, and   for higher dimensional K\"ahler manifolds by  Sibley \cite{Sib}  and Li-Zhang-Zhang \cite{LZZ}.

In the case of non-K\"ahler manifolds, it  is of interest to  establish   existence of the Yang-Mills flow in vector bundles over $n$-dimensional Riemannian manifolds. In lower dimensions (i.e. $n=2$ or $3$), Råde \cite {Ra} established global existence of smooth solutions of the Yang-Mills flow in $M\times [0,\infty)$. For $n=4$,   Struwe \cite {St3} proved the
existence of weak solutions to the Yang-Mills flow, with initial connection $A_0\in H^{1}(M)$, that are gauge-equivalent to smooth solutions  in $M\times (0, T_1)$  for some maximal existence time $T_1>0$, which   is analogous to the one  for harmonic maps  in \cite{St1}. Using Uhlenbeck's broken Hodge gauge \cite{U1},  Hong-Schabrun \cite {HS}  established an energy identity  for the Yang-Mills flow at the maximal existence time $T_1$. For smooth initial connections,   Schlatter-Struwe-Tahvildar-Zadeh \cite {SST} proved  global existence of the
$\text{SO}(4)$-equivariant Yang-Mills flow on $\mathbb R^4$.  Hong-Tian
\cite {HT2} proved global existence of the $m$-equivariant Yang-Mills flow on $\mathbb R^4$. Using the Yang-Mills $\alpha$-flow, Hong-Tian-Yin \cite{HTY} proved global existence of weak solutions of the Yang-Mills flow on four manifolds. Waldron  \cite {W2} (see also \cite {W1}) then solved the problem on four manifolds and established global existence of smooth solutions of the Yang-Mills flow with smooth initial connection. However,
when $n\geq 5$,  the Yang-Mills flow  blows up in finite time as demonstrated by Naito \cite{N} (see also \cite{GS}).

In this paper, we investigate the Yang-Mills flow over compact $n$-dimensional Riemannian manifolds  for $n \geq 4$. Particularly, we will prove the maximal existence of the Yang-Mills flow with non-smooth initial connections $A_0\in W^{1,\frac{n}{2}}(M)$.  For an initial connection  $A_0 \in W^{1,\frac{n}{2}}(M)$ , there exists a sequence of smooth connections $A_{k}$ strongly converging to $A_0$ in $W^{1, {\frac{n}{2}}}(M)$ (see, e.g., \cite{U2} ). To establish existence with non-smooth initial connections, one expects that a sequence of smooth solutions of the Yang-Mills flow with initial connection  $A_{k}$ will converge to a limit weak solution. For this reason, we develop a parabolic formulation of the gauge fixing theorem of the Yang-Mills flow as follows.

\begin{thm}\label{Theorem 1.1}  For $n\geq 4$, let $D_A$ be a smooth   solution  of the Yang-Mills
flow (\ref{YMF}) in $\bar B_{r_0}(x_0)\times [0,t_1]$ with smooth initial connection $A_0$ for some  $t_1>0$, where $B_{r_0}(x_0)$ is an open ball in $M$ with centre at $x_0$   and radius $r_0>0$. Assume that there exists  a sufficiently small $\varepsilon >0$ such that
 \begin{equation}\label{small-condition}
 \sup_{0\leq t\leq t_1} \int_{B_{r_0}(x_0)} |F_A(x,t)|^{n/2}dv \leq \varepsilon.  \end{equation}
 Then, for any $t\in [0,t_1]$, there exists smooth gauge
transformations $S(t)\in G$  such that $D_A$ is gauge-equivalent to
$D_{a(t)}=S^*(D_{A(t)})$ satisfying
 \begin{equation}\label{US1}  d^*  a(t)=0 \quad \mbox {in } B_{r_0}(x_0),\quad a(t)\cdot \nu=0 \mbox { on }\partial  B_{r_0}(x_0),\end{equation}
\begin{equation}\label{US2}\int_{B_{r_0}(x_0)} \frac {|a(t)|^{n/2}}{r^{n/2}_0}+|\nabla a(t)|^{n/2}\,dv\leq C \int_{B_{r_0}(x_0)} |F_{a(t)}|^{n/2}\,dv. \end{equation}
    Furthermore,  $D_a=S^*(D_A)$ is a solution to the equivalent Yang-Mills flow
\begin{eqnarray}\label{good}
    \pfrac{a}{t}=-D_{a}^*F_{a}+D_{a} s\quad\mbox { in }B_{r_0}(x_0)\times [0,t_1]
    \end{eqnarray}
    for which
   \begin{equation}\label{NU}  \int_{0}^{t_1 }\int_{B_{r_0}(x_0)}  \frac 1{r^2_0}|s|^2+|D_{a} s|^2 +| \pfrac{a}{t}|^2 \,dv\,dt\leq C \int_{0}^{t_1 } \int_{B_{r_0}(x_0)} |\nabla_a
    F_a|^2\,dv\,dt  \end{equation}
    with $s(t)=S^{-1}(t)\circ \frac d{dt} S(t)$.

\end{thm}

We would like to remark that the equivalent Yang-Mills flow \eqref{good} with \eqref{US1} is parabolic although the Yang-Mills flow \eqref{YMF} itself is not.
For the proof  of Theorem \ref {Theorem 1.1}, we  modify a method of Uhlenbeck \cite{U2}.
Using Uhlenbeck's gauge fixing theorem,
  $D_A$ is gauge-equivalent to $D_{a}=S^{-1}\circ D_A \circ S$ via a gauge transformation $S\in G$
  satisfying  $d^*a=0$ in $B_{r_0}(x_0)$  with
   normal  boundary condition $a\cdot \nu =0$ on $\partial B_{r_0}(x_0)$. However, the gauge transformation
 $S$ has no proper boundary condition
   on $\partial B_{r_0}(x_0)$,  so we cannot
    prove (\ref{NU}) directly from \cite{U2}. To overcome this, we observe that
    the  gauge transformation $S$ in the proof of Uhlenbeck  \cite{U2} is a product of two gauge transformations $S_1=e^{u_1}$
    and $S_2=e^{u_2}$, for which $u_1$ satisfies the Dirichlet boundary condition $u_1=0$ on $\partial B_{r_0}(x_0)$ and $u_2$ satisfies
   the Neumann boundary condition $u_2\cdot\nu =0$ on $\partial B_{r_0}(x_0)$.
Using these two boundary conditions, we extend the result of Uhlenbeck \cite{U2} to prove Theorem \ref{Theorem 1.1}.

Due to Theorem \ref {Theorem 1.1}, we can generalize the result of Struwe \cite {St3} on the existence of weak solutions of the Yang-Mills flow in four dimensions to higher dimensions.

\begin{thm}\label{Theorem 1.2} For $n\geq 4$,  let $D_{A_0}$  be a  connection satisfying $A_0 \in W^{1,\frac{n}{2}}(M)$ and $F_{A_0}\in L^{n/2}(M)$.
   Then there exists a weak solution of the Yang-Mills flow (\ref{YMF}) with initial connection $A_0$ in $M\times [0, T_1)$, where $T_1$ is  the maximal existence time.  At $t = T_1$, there exists constants $\varepsilon_1>0$,    $R_0>0$ and a point $x_0\in M$ such that
\begin{equation}  \label{Cri}
\limsup_{t \to T_1}  \int_{B_R(x_0)} |F_A(x, t)|^{n/2}\,dv \geq \varepsilon_1
\end{equation}
 for any $R\in (0, R_0]$. Furthermore, for each $t\in  (0,T_1)$, the solution $A(t)$ is gauge-equivalent to a smooth solution of the Yang-Mills flow.
\end{thm}
To prove Theorem \ref{Theorem 1.2}, we use Theorem \ref{Theorem 1.1}  to establish a compactness theorem on the Yang-Mills flows (see below Theorem \ref{Theorem 4.1}). From this compactness theorem,   we can demonstrate local existence of the Yang-Mills flow with initial connection $A_0\in W^{1,\frac{n}{2}}(M)$ with curvature $F_{A_0}\in L^{n/2}(M)$. The criterion \eqref{Cri} of the maximal existence time follows from two new derivations. Firstly,
 assuming that for  each $T<T_1$,    there is some constant  $R_0>0$ such that
 \begin{equation}  \label{Re1}
 \sup_{0\leq t\leq T} \max_{x_0\in M} \int_{B_{R_0}(x_0)} |F_A(\cdot ,t)|^{n/2}dv < \varepsilon_1 ,
 \end{equation}
     we apply the idea from $n$-harmonic maps of Hungerb\"uhler \cite{Hu} (see also \cite{Ho2}) to modify Moser's iterations and  prove that the curvature $F_A$ is bounded in $M\times [0,T]$, without using Schoen's $\varepsilon$-regularity theorem \cite{Schoen}. Secondly, using the local uniform bound on the curvature $F_A$, we use the idea of Hamilton \cite{Ha} on the Ricci flow to obtain   uniform  estimates on $\nabla^k F_A$  for $k\geq 1$. We note that this proof  is different from the one in Hong-Tian \cite{HT1} that used the Bochner formula of $\nabla^k_A F_A$ from Donaldson \cite{Do1}.

\begin{rem}    When $n=4$, Struwe \cite {St2} established   uniqueness of weak solutions of the Yang-Mills flow under an extra  condition that $A_0$ is irreducible. Since then,  it has been  a challenging   problem to prove the uniqueness of  weak solutions of the Yang-Mills flow in four manifolds with   initial data  $A_0\in W^{1,2}(M)$ (see, e.g., \cite {W1}).   In a separate paper \cite{CH},    Chan and Hong proved that the weak solutions   are unique for $n=4$.
For $n\geq5$, Gastel \cite{Ga} showed non-uniqueness of weak solutions to  the Yang–Mills flow.
	 \end{rem}

  Furthermore, we investigate the blow-up of the Yang-Mills flow at the singular time $t=T_1$ as following:

  \begin{thm}  \label{Theorem1.3'}
    Let $A(t)$ be a smooth solution of the Yang-Mills flow with initial connection $A_0$ in $M \times [0,T_1)$, where $T_1$ is the maximal existence time in \eqref{Cri}. Then there exist sequences $R_k \rightarrow 0$, $t_k \rightarrow T_1$ and points $x_k \to x_0$ such that as $k \rightarrow \infty$,
\begin{align}
A_k(x,t)=R_k A(x_k+R_k x, t_k +R_k^2 t) \rightarrow  \bar {A}(x,t),
\end{align}
in $C^{\infty}_{loc}(\Bbb R^n\times (-\infty, 0])$, where $\bar {A}(x,t)$ is a nontrivial smooth solution of the Yang-Mills flow in $\Bbb R^n\times (-\infty, 0]$ satisfying
\begin{equation}
     \frac{\partial  \bar  A  }{\partial   t}-\frac{  x_\alpha }{2| t|} F^{\alpha\beta}_{ \bar A }d x_{\beta}=0.
\end{equation}
 \end{thm}

We can show the global existence of YM flow:
\begin{thm}\label{Theorem 1.3}

   Assume that there exists a sufficiently small constant $\varepsilon_3 >0$ such that
\begin{equation}  \label{GB1}
\int_{M} |F_{A_0}|^{\frac n 2}\,dv < \varepsilon_3.  \end{equation}
Then, there exists a global
  smooth solution of the Yang-Mills flow (\ref{YMF}) with initial value $A_0$ in $M\times [0, \infty)$.
  \end{thm}

To prove this, we need a local   monotonicity formula for the
Yang-Mills flow and an improved $\varepsilon$-regularity theorem. Although the results were outlined by Hong-Tian \cite {HT1},
we provide detailed proofs of  a local   monotonicity formula for the
Yang-Mills flow and obtain an improved $\varepsilon$-regularity theorem, that
if there exists a constant $\varepsilon_0 > 0$ such
that
\begin{equation}  \label{EP}
R^{2-n}\int_{P_R(x_0, t_0)}|F_A|^2\,dv\,dt
    \leq \varepsilon_0\end{equation}
     for  some small $R \ > 0$,
  then we have
$$\sup_{P_{\delta R}(x_0,t_0)} |F_A|^2 \leq C(\delta R)^{-4}\,,$$
where $0<\delta<1$.  We then apply this $\varepsilon$-regularity theorem to prove Theorem \ref{Theorem 1.3}.

Finally, we improve a result of Hong-Tian \cite{HT1}, that was used to settle the Bando-Siu conjecture (see \cite{Sib} and \cite{LZZ}) over compact K\"ahler manifolds, and for the case of Riemannian manifolds only applied to a particular sequence of $t_i\to\infty$.

  \begin{thm}\label{Theorem 1.4}  Let $A(x,t)$ be a global smooth solution of
the Yang-Mills flow  (\ref{YMF})  in $M\times [0,\infty )$. Then for any sequence $t_i\to\infty$, there exists a subsequence $\{t_{i_k}\}$   such
that, as $t_{i_k}\to\infty$, $A(x, t_{i_k})$ converges, modulo gauge
transformations, to a Yang-Mills connection $A(x)$ in  smooth
topology outside  a closed set $\Sigma \{A(t_{i_k})\}$ given  by
\begin{equation}  \label{Conc}
\Sigma \{A(t_{i_k})\}=\bigcap_{R\leq R_0} \left\{x\in M:\quad
\liminf_{k\to\infty}\, R^{4-n}\int_{B_R(x)}   |F_A|^2  (\cdot
,t_{i_k})|^2\,dv\geq \varepsilon_3\right \}\,\end{equation}
 for some $\varepsilon_3 >0$
 and $R_0>0$. Furthermore,  $\Sigma \{A(t_{i_k})\}$ is $(n-4)$-rectifiable.
   \end{thm}

 In a special case of  holomorphic vector bundles over  compact K\"ahler manifolds, Theorem \ref{Theorem 1.4} was firstly proved by Hong-Tian \cite{HT1},   which was used to settle    the  Bando-Siu conjecture   on  K\"ahler
manifolds (see, e.g., \cite{Sib}). For a general case   of  vector bundles over  Riemanian manifolds, Hong-Tian \cite{HT1} only proved Theorem \ref{Theorem 1.4}   for a special sequence of $t_i\to\infty$.
Recently, Waldron \cite{W3} improved the result     for any sequence $t_i\to\infty$ using an alternative method. However, the concentrated set $\Sigma$ defined in \cite{W3} has a more complicated form than the form in \eqref{Conc}.   Using a modification of the method in \cite{HT1},  we prove   Theorem \ref{Theorem 1.4} any sequence of $t_i\to \infty$.

  The paper has been organised as follows. In Section 2, we give a priori estimates and present new proofs of the Yang-Mills flow. In Section 3, we develop a parabolic formulation of the gauge fixing theorem. In Section 4, we  prove  Theorem \ref{Theorem 1.2}. In Section 5, we prove Theorems \ref{Theorem1.3'}-\ref{Theorem 1.4}.

\medskip
\noindent{\bf AI Disclosure:}{
The first version of this paper was finished in 2025.
 The authors     developed and verified all proofs  without using AI, and   take full responsibility for the content of the manuscript.
}

\section{Notation and new  estimates   on  the Yang-Mills  flow}

We first provide some notation and results of the Yang-Mills  flow. For an integer $k\geq 0$, set   $\Omega ^k(Ad \,E ) := \Gamma
(Ad \,E \otimes\Omega ^k(M))$.
In local coordinates,   $a\in \Omega ^k(Ad \,E )$ is  represented   by
$$  a(x) =  \sum_{i_1<\cdots <i_k} a_{i_1, \cdots, i_k}(x)dx_{i_1}\wedge \cdots  \wedge dx_{i_k} $$
with $a_{i_1, \cdots, i_k}(x)\in \mathfrak{g}$.
For $a, b\in \Omega ^k(Ad \,E )$, the inner product $\<a,b\>$  is defined by
$$  \< a, b\> :=
    - \sum \text{Trace}\(a_{i_1,\cdots, i_k}b^*_{j_1, \cdots, j_k}\)
     \< dx_{i_1}\wedge  \cdots \wedge  dx_{i_k},
      dx_{j_1}\wedge \cdots  \wedge  dx_{j_k}\> , $$
where $b^*_{j_1, \cdots, j_k}$ is the  conjugate transpose of   $b_{j_1, \cdots, j_k}$.

Let $D_{\text{ref}}$ be a given smooth reference connection in $E$.  For a connection $D_{A}$, we write  $D_A:=D_{\text{ref}} +A$  with $A\in \Omega^1(\text{Ad} \,E )$.
Each connection  $A$ is  compatible with the Riemannian structure of $M$; i.e.
 for
 an exterior derivative
$$D_A : \Omega ^k(Ad\,E) \to \Omega ^{k+1}(Ad \, E)  \;\;
\mbox {for an integer} \;\;k \geq 0,  $$
there exists an adjoint operator
$$D_A^*: \Omega^{k+1}(\text{Ad} \ E )\to \Omega^{k}(\text{Ad} \ E)$$
such that
$$ \int _M \<D_Aa, b\>\,dv=\int _M \<a, D_A^*b\>\,dv$$
for all $a\in \Omega^{k}(\text{Ad} \ E)$ and $b\in \Omega^{k+1}(\text{Ad} \ E )$.
With the above fibrewise inner product,   we  can define
$L^2\(\Omega ^p(Ad \,E )\)$ as the completion of $C^\infty $ sections
under the $L^2$-norm.

Let  $D_{A_0}$ be a  given smooth connection in $E$. For a connection $D_{a(t)}$, we write $D_{a(t)}=D_{\text{ref}}+a(t)$. Then
\[  F_{D_{a(t)}} = F(D_{\text{ref}}) +D_{\text{ref}}a(t) +a(t)\#a(t).\]
Following from \cite{St2}, we can construct an equivalent flow
\begin{equation}\label {2.1}
    \pfrac{a(t)}{t}=- D_{a(t)}^*F_{D_{a(t)}}- D_{a(t)}( D_{a(t)}^*a)
\end{equation}
with $a(0)=A_0$, and  $D_{A_0}=D_{\text{ref}} +A_0$ a given smooth connection in $E$. We note that  (\ref{2.1})  is a nonlinear parabolic
system. Subsequently, according to the well-known theory of partial differential
equations, there exists a unique smooth solution of (\ref{2.1}) with
the initial connection $a(0)$ on $M\times[0,t_1]$ for some $t_1>0$. Following from the
theory of ordinary differential equations, there exists a unique smooth
solution $S\in G$ to the following initial problem:
 \begin{equation}\label {2.2} \frac d{dt} S=-S\circ (D_{a}^*a) \quad \mbox{in } M\times [0, t_1]
\end{equation}
with initial connection  $S(0)=I$ for some $t_1>0$.

Through
 a gauge transformation, we can determine that $S^{-1}(t)\in G$, $ D_{a(t)}$ is gauge equivalent to
 \begin{equation}\label {2.3} D_{A(t)}=(S^{-1})^*(D_{a(t)})=S\circ   D_{a(t)}  \circ S^{-1}   \end{equation}
satisfying
\[F_A=S F_{D_{a}} S^{-1}.\]
Combining (\ref{2.1}), (\ref{2.2}) with  (\ref{2.3}) yields
  \begin{eqnarray*}\frac d{dt} D_A&=& \frac {dS}{dt} \circ  D_{a(t)} \circ S^{-1}
  +S\circ \frac {d  D_{a(t)}}{dt} \circ
  S^{-1} +
S\circ  D_{a(t)} \circ \frac {dS^{-1}}{dt}  \\
&=&S^{-1}\left ( - D_{a(t)}^*F_{D_{a(t)} }\right )S\\
    &=&-D_A^*F_A
 \end{eqnarray*}
 (e.g.  \cite {St2}, \cite {Ho1}).
 This shows that
$D_A=(S^{-1})^* (D_{a(t)})$ satisfies the Yang-Mills flow with
$A(0)=A_0$ in $M\times [0,t_1]$ for some $t_1>0$ and is unique (see \cite{HTY}).

\begin{lem}  \label{Lemma 1} Let $A(t)$ be a smooth solution to the Yang-Mills flow (\ref{YMF}) in $M\times
 [0,T]$ with initial connection $A(0)=A_0$ for some $T>0$. Then, for $0<s\leq T$, we have
\begin{equation} \label{Energy identity}
  \int_M |F_{A(s)}|^2  \,dv +2
  \int_0^s \int_M  \abs{\pfrac{A}{t}}^2 dv\,dt= \int_M  |F_{A_0}|^2
  \,dv
\end{equation}
and
 \begin{equation} \label{Loca identity}
    \int_{B_{R_0}(x_0)}| F_{A(s)}|^2 \, dv\leq \int_{B_{2R_0}(x_0)} |F_{A_0}|^2
  \,dv + \frac C {R_0^2}\int_0^s  \int_{B_{2R_0}(x_0)} |F_A|^2 \, dv\,dt.
 \end{equation}
\end{lem}
\begin{proof}
  Let $\phi\in C_0^{\infty}(B_{2R_0}(x_0))$ be a cutoff function with $\phi =1$ in $B_{R_0}(x_0)$ and $|\nabla \phi|\leq \frac C{R_0}$.
Recall from \cite{St2} that
\[\frac {\partial F_A}{\partial t}=-D_AD_A^*F_A.\]
Then using Young's inequality, we have
 \begin{eqnarray*}
&&\frac 1 2 \frac{d}{dt}\int_M | F_A|^2\phi^2\, dv=\int_M \left <\frac {\partial F_A}{\partial t}, F_A\right >\phi^2\, dv\\
&=& -\int_M |D_A^*F_A|^2\phi^2 \, dv +\int_M \left <D_A^*F_A, F_A\right >\#\nabla \phi^2 \, dv\\
&\leq & -\frac 1 2\int_M |D_A^*F_A|^2\phi^2 \, dv +\int_M |F_A|^2 |\nabla \phi |^2 \, dv.
 \end{eqnarray*}
Subsequently, \eqref {Loca identity} follows from integrating with respect to $t$. Choosing $\phi=1$, we obtain \eqref{Energy identity}.
\end{proof}

Given a connection $D_A$, we define the Hodge Laplacian
$$  \Delta_A := D_A^*D_A + D_AD_A^*  , $$
 on $\Omega ^p(\text{Ad} \ E )$ for $p \geq 0$. We also have the ``trace Laplacian'' $\D_A^*\D_A$ corresponding to the
covariant derivative $\D_A$.  The relation between these operators is given by
\begin{equation}\label{Wz}
 \D_A^*\D_A\phi =  \Delta_A \phi +R_M\#\phi+ F_A\#\phi \,,
 \end{equation}
where $\phi \in \Omega ^p(\text{Ad} \ E )$,
  $R_M$ is the Riemannian curvature of $M$ and
$\#$ denotes  a multi-linear map with smooth coefficients   (see \cite{HT1}). Equation
\eqref{Wz} is the so-called Weitzenb\"ock formula, from which the following lemmas are derive from.

\begin{lem}  \label{Lemma 2} Let $A(t)$ be a smooth solution to the Yang-Mills flow (\ref{YMF}) in $M\times
 [0,T]$ with initial connection $A(0)=A_0$. Assume that there are constants  $\varepsilon>0$ and $0 < R_0 < 1$ such that
 \[\sup_{0\leq t\leq T} \max_{x_0\in M} \int_{B_{R_0}(x_0)} |F_A(\cdot ,t)|^{n/2}dv \leq \varepsilon .\]
 Then there exists a constant $C > 0$ such that
 \begin{equation}\label{T2}
    \int_0^T \int_M|\nabla_A F_A|^2\,dv\, dt\leq C\left(1+\frac {T}{R_0^2}\right)\int_M|F_{A_0}|^2\, dv.
\end{equation}
\end{lem}
\begin{proof}
It follows from the Bianchi identity  $D_A F_A=0$ and the Weizenb\"ock formula \eqref{Wz} that
\begin{equation}\label{WZ}
  D_AD_A^*F_A=\nabla_A^*\nabla_A F_A+ F_A\#F_A+ R_M\#F_A.
\end{equation}

Let $\{B_{R_0}(x_i)\}_{i=1}^J$ be a standard open cover of $M$ with a finite subcover at each point. Using the H\"older inequality and the Sobolev inequality, we get
\begin{eqnarray*}
  &&\int_M |\nabla_A F_A|^2 dv \leq  \int_M |D^*_AF_A|^2 dv +C\int_M |F_A|^3+ |F_A|^2 dv \\
  &\leq & C\sum_{i=1}^J\left (\int_{B_{R_0}(x_i)} |F_A|^{n/2}dv\right )^{\frac 2 n}\left (\int_{B_{R_0}(x_i)} |F_A|^{\frac {2n}{(n-2)}}dv\right )^{\frac {(n-2)}{n}}\\
  &&+C\int_M |D^*_AF_A|^2 + |F_A|^2 dv \\
  &\leq & C\varepsilon^{2/n} \int_M|\nabla_AF_A|^2\,dv +C\int_M\left(1+\frac {1}{R_0^2}\right)|F_A|^2+|D^*_AF_A|^2 \, dv.
\end{eqnarray*}
Integrating and choosing $\varepsilon > 0$ sufficiently small, we obtain \eqref{T2}.
\end{proof}

\begin{lem}  \label{Lemma 3} Let $A(t)$ be a smooth solution to the Yang-Mills flow (\ref{YMF}) in $M\times [0,T]$. There exists a constant
  $\varepsilon =\varepsilon (E) >0$ such that if
\[\sup_{0\leq t\leq T} \max_{x_0\in M} \int_{B_{R_0}(x_0)} |F_A(\cdot ,t)|^{n/2}dv \leq \varepsilon  \]
for some $R_0 > 0$. Then there exists a constant $C > 0$ such that
 \begin{eqnarray} \label{local energy}
    &&\qquad\int_{B_{R_0}(x_0)} |F_A|^{n/2} (\cdot ,t) dv +\int_0^t\int_{B_{R_0}(x)}  |F_A|^{\frac {n-4}2} |\nabla_A F_A|^2 \,dv\, ds\\&&\leq \int_{B_{2R_0}(x)} |F_A|^{n/2}  (0) dv +  \frac{C}{R_0^2} \int_0^t \int_{B_{2R_0}(x_0)} |F_A|^{n/2} (s) dv\, ds\nonumber
  \end{eqnarray} for any $x_0\in M$ and $t\in [0, T]$.
\end{lem}
\begin{proof} Let $\phi\in C_0^{\infty}(B_{2R_0}(x_0))$ be a cutoff function with $\phi =1$ in $B_{R_0}(x_0)$ and $|\nabla \phi|\leq \frac C{R_0}$.
It follows from  (\ref{YMF}) and the Weizenb\"ock formula (\ref{WZ}) that
\begin{equation}\label{W2}
\frac {\partial F_A}{\partial t}=-D_AD_A^*F_A=-\nabla_A^*\nabla_A F_A+ F_A\#F_A+ R_M\#F_A.
\end{equation}
Using \eqref{W2},  we   derive
\begin{eqnarray*}
  && \frac {d}{dt}  \int_{B_{2R_0}(x)} |F_A|^{n/2}\phi^2\,dv=\frac n 2\int_{B_{2R_0}(x)} |F_A|^{\frac {n-4}2}\left <F_A, \, \frac {\partial F_A}{\partial t}\right >\phi^2\,dv\\
  &&=-\frac n 2\int_{B_{2R_0}(x)} \left < \nabla_A(|F_A|^{\frac {n-4}2}F_A\phi^2), \,  \nabla_AF_A\right >\,dv\\
  &&\quad+\frac n 2\int_{B_{2R_0}(x)} |F_A|^{\frac {n-4}2}\left <F_A, \, F_A\#F_A+ R_M\#F_A \right>\phi^2\,dv\\
  && \leq -\frac n 2\int_{B_{2R_0}(x)}  (|F_A|^{\frac {n-4}2} |\nabla_A F_A|^2 +\frac {n-4}2 |F_A|^{\frac {n-4}2}|\nabla |F_A||^2) \phi^2\,dv \\
  &&\quad +C\int_{B_{2R_0}(x)} |F_A|^{\frac {n-4}2} (|F_A|^3+ |F_A|^{2}   + \varepsilon  |\nabla_A F_A|^2 )\phi^2 +  |\nabla\phi |^2 \,|F_A|^{n/2}\,dv.
  \end{eqnarray*}
 Now, since $B_{2R_0}(x_0)$ is covered by a fixed finite number of $B_{R_0}(x_i)$, it follows from H\"older's inequality that
  \begin{eqnarray*}
  &&  \int_{B_{2R_0}(x)} |F_A|^{\frac {n-4}2} |F_A|^3\phi^2 \,dv=\int_{B_{2R_0}(x)}  |F_A||F_A|^{\frac {n}2}\phi^2 \,dv\\
  &&\leq \left (\int_{B_{2R_0}(x)} |F_A|^{n/2}dv\right )^{\frac 2 n}\left (\int_{B_{2R_0}(x)} |F_A|^{\frac {n^2}{2(n-2)}}\phi^{\frac {n}{n-2}} dv\right )^{\frac {n-2}{n}}\\
  &&\leq C\varepsilon^{\frac 2 n}\int_{B_{2R_0}(x)} |\nabla (|F_A|^{\frac {n}{4}}\phi) |^2 dv\\
  &&\leq C\varepsilon^{\frac 2 n}\int_{B_{2R_0}(x)} |F_A|^{\frac {n-4}{2}}|\nabla |F_A| |^2\phi^2 dv+ C\int_{B_{2R_0}(x)} |F_A|^{\frac {n}{2}}|\nabla \phi  |^2 dv.
  \end{eqnarray*}
  Combining the above inequalities and choosing $\varepsilon > 0$ sufficiently small,  we obtain \eqref{local energy}.
\end{proof}

For each $R>0$, let
$$P_R(x_0, t_0) :=B_R(x_0)\times [t_0-R^2, t_0 ].$$
\begin{lem}  \label{Lemma 4} Let $A(t)$ be a smooth solution to the Yang-Mills flow (\ref{YMF}) in $M\times [0,T]$. There exists a  sufficiently small constant
  $\varepsilon =\varepsilon (E) >0$ such that
if
\[ \sup_{t_0-4R^2_0\leq t\leq t_0}\int_{B_{2R_0}(x_0)} |F_A(\cdot ,t)|^{n/2}dv \leq \varepsilon  \]
for some $R_0 > 0$. Then
 \begin{equation} \label{global}
     \int_{P_{R_0}(x_0, t_0)} |F_A(x ,t)|^{\frac n 2 +2}dv\,dt  \leq  C\(1+\frac 1{R_0^{n}}\) \int_{P_{2R_0}(x_0)} |F_A(x ,t)|^{2}dv\,dt,
  \end{equation}  for some constant $C > 0$ independent of $R_0$ and the connection $A$.
\end{lem}

\begin{proof} In a neighborhood of each point $x_0\in M$,  we denote by $\nabla_i$   the
first covariant derivative with respect to $e_i$ and  by $\nabla^2_{ij}$
the second covariant derivatives, and so on.
Let $\phi=\phi (x,t)$ be a cut-off function with support in  $B_{2R_0}(x_0)\times [t_0-4R_0^2, t_0]$ such that $\phi =1$ in $B_{R_0}(x_0)\times [t_0-R_0^2, t_0]$,
$|\nabla \phi|\leq C/R_0$, $|\partial_t \phi|\leq \frac 1 {R_0^2}$ and $|\phi|\leq 1$ in  $B_{R_0}(x_0)\times [t_0-4R_0^2, t_0]$.

Recall that
\begin{equation}\label{W2.1}
\frac {\partial F_A}{\partial t}=-D_AD_A^*F_A=-\nabla_A^*\nabla_A F_A+ F_A\#F_A+ R_M\#F_A.
\end{equation}
Multiplying by $|F_A|^{\frac n2-1}F_A\phi^n$,
integrating by parts, and applying Young's inequality, we obtain
\begin{equation}\label{2.13}
\begin{aligned}
&\quad \frac{2}{n+2}\frac{d}{dt}
\int_M |F_A|^{\frac n2+1}\phi^n\,dv
+\frac12
\int_M |F_A|^{\frac n2-1}
|\nabla_A F_A|^2\phi^n\,dv
\\
&\leq
C\int_M |F_A|^{\frac n2+2}\phi^n\,dv
+C\int_M |F_A|^{\frac n2+1}
\left(
\phi^n
+\phi^{n-2}|\nabla\phi|^2
+\phi^{n-1}|\partial_t\phi|
\right)\,dv
\\
&\leq
C\int_M |F_A|^{\frac n2+2}\phi^n\,dv
+C\left(1+\frac1{R_0^2}\right)
\int_M |F_A|^{\frac n2+1}\phi^{n-2}\,dv.
\end{aligned}
\end{equation}
By H\"older's inequality and the Sobolev inequality, we have
\begin{equation}\label{2.14}
\begin{aligned}
&\quad \int_M |F_A|^{\frac n2+2}\phi^n\,dv
\\
&=
\int_{B_{2R_0}(x_0)}
|F_A|
\left(
|F_A|^{\frac{n+2}{4}}\phi^{\frac n2}
\right)^2\,dv
\\
&\leq
\left(
\int_{B_{2R_0}(x_0)}
|F_A|^{\frac n2}\,dv
\right)^{\frac2n}
\left[
\int_M
\left(
|F_A|^{\frac{n+2}{4}}\phi^{\frac n2}
\right)^{\frac{2n}{n-2}}\,dv
\right]^{\frac{n-2}{n}}
\\
&\leq
C\varepsilon^{\frac2n}
\left[
\int_M
\left|
\nabla\left(
|F_A|^{\frac{n+2}{4}}\phi^{\frac n2}
\right)
\right|^2\,dv
+
\int_M |F_A|^{\frac n2+1}\phi^n\,dv
\right]
\\
&\leq
C\varepsilon^{\frac2n}
\int_M |F_A|^{\frac n2-1}
|\nabla_A F_A|^2\phi^n\,dv
+C\varepsilon^{\frac2n}
\left(1+\frac1{R_0^2}\right)
\int_M |F_A|^{\frac n2+1}\phi^{n-2}\,dv.
\end{aligned}
\end{equation}
Substituting \eqref{2.14} into
\eqref{2.13} and choosing $\varepsilon>0$
sufficiently small, we obtain
\begin{equation}
\begin{aligned}
&\quad \frac{2}{n+2}\frac{d}{dt}
\int_M |F_A|^{\frac n2+1}\phi^n\,dv
+\frac14
\int_M |F_A|^{\frac n2-1}
|\nabla_A F_A|^2\phi^n\,dv
\\
&\leq
C\left(1+\frac1{R_0^2}\right)
\int_M |F_A|^{\frac n2+1}\phi^{n-2}\,dv.
\end{aligned}
\end{equation}
Integrating in time yields
\begin{equation}\label{eq:integrated-gradient}
\begin{aligned}
&\int_{P_{2R_0}(x_0,t_0)}
|F_A|^{\frac n2-1}
|\nabla_A F_A|^2\phi^n\,dv\,dt
\leq
C\left(1+\frac1{R_0^2}\right)
\int_{P_{2R_0}(x_0,t_0)}
|F_A|^{\frac n2+1}\phi^{n-2}\,dv\,dt.
\end{aligned}
\end{equation}
Integrating \eqref{2.14} in time and applying
\eqref{eq:integrated-gradient}, we obtain
\begin{equation}\label{eq:integrated-curvature}
\begin{aligned}
&\int_{P_{2R_0}(x_0,t_0)}
|F_A|^{\frac n2+2}\phi^n\,dv\,dt
\leq
C\left(1+\frac1{R_0^2}\right)
\int_{P_{2R_0}(x_0,t_0)}
|F_A|^{\frac n2+1}\phi^{n-2}\,dv\,dt.
\end{aligned}
\end{equation}
Young's inequality implies that
\begin{equation}
\begin{aligned}
&\int_{P_{2R_0}(x_0,t_0)}
|F_A|^{\frac n2+2}\phi^n\,dv\,dt
\\
&\leq
C\left(1+\frac1{R_0^2}\right)^{\frac n2}
\int_{P_{2R_0}(x_0,t_0)}
|F_A|^2\,dv\,dt\\
&\leq
C\left(1+\frac1{R_0^n}\right)
\int_{P_{2R_0}(x_0,t_0)}
|F_A|^2\,dv\,dt.
\end{aligned}
\end{equation}
\end{proof}

\begin{lem}  \label{Lemma  5}   Let $A(t)$ be a smooth solution to the Yang-Mills flow (\ref{YMF}) in $B_{r_0}(x_0)\times
 [t_0-4R_0^2,t_0]$.  Assume that there exists a sufficiently small $\varepsilon >0$ such that
 \[\sup_{t_0-4R_0^2\leq t\leq t_0}  \int_{B_{2R_0}(x_0)} |F_A(x,t)|^{n/2}dv \leq \varepsilon .\]
 Then there exists a constant $C > 0$  independent of the connection $A$ such that
\begin{eqnarray*}
          &&   |F_A(x_0, t_0)|^2\leq  C\(\frac 1 {R_0^4}+\frac 1{R_0^{n+2}} \int_{P_{2R_0}(x_0)} |F_A(x ,t)|^{2}dv\,dt \).
    \end{eqnarray*}
\end{lem}
\begin{proof} Without loss of generality, we assume that $R_0=1$ by  re-scaling.
For constants $\rho > 0$, $R > 0$ with $\rho <R\leq R_0=1$,
 let $\phi=\phi (x,t)$ be a cut-off function with support in  $B_{R}(x_0)\times [t_0-R^2, t_0]$ such that $\phi =1$ in $B_{\rho }(x_0)\times [t_0-\rho^2,
t_0]$, $|\nabla \phi|\leq \frac C{R-\rho}$, $|\partial_t \phi|\leq \frac 1 {(R-\rho)^2}$ and $|\phi|\leq 1$ in  $B_{R}(x_0)\times [t_0-R^2, t_0]$.

Multiplying \eqref{W2.1} by $(1+|F_A|)^{\beta+1}\phi^2$, similar computation as in \eqref{2.13} yields
\begin{equation}
\begin{aligned}
&\quad \frac{1}{\beta+2}\frac{d}{dt}
\int_{B_R(x_0)}
(1+|F_A|)^{\beta+2}\phi^2\,dv+
\frac{\beta+1}{2}
\int_{B_R(x_0)}
(1+|F_A|)^\beta
|\nabla|F_A||^2\phi^2\,dv
\\
&\leq
C\int_{B_R(x_0)}
(1+|F_A|)^{\beta+3}\phi^2\,dv+
\frac{C}{\beta+2}
\int_{B_R(x_0)}
(1+|F_A|)^{\beta+2}
\left(
|\phi\partial_t\phi|+|\nabla\phi|^2
\right)\,dv.
\end{aligned}
\end{equation}
Integrating in time, we get
\begin{equation}\label{2.20}
\begin{aligned}
&\sup_{t_0-R^2\leq t\leq t_0}
\int_{B_R(x_0)}
(1+|F_A|)^{\beta+2}\phi^2\,dv
+
\int_{P_R(x_0,t_0)}
\left|
\nabla\left(
(1+|F_A|)^{\frac{\beta+2}{2}}\phi
\right)
\right|^2\,dv\,dt
\\
&\quad+
\int_{P_R(x_0,t_0)}
(1+|F_A|)^{\beta+2}\phi^2\,dv\,dt
\\
&\leq
C(\beta+2)
\int_{P_R(x_0,t_0)}
(1+|F_A|)^{\beta+3}\phi^2\,dv\,dt\\
&\quad
+
C\int_{P_R(x_0,t_0)}
(1+|F_A|)^{\beta+2}
\left(
|\phi\partial_t\phi|+|\nabla\phi|^2
\right)\,dv\,dt.
\end{aligned}
\end{equation}
 Applying H\"older's and Sobolev's inequalities, we get
\begin{equation}\label{2.21}
\begin{aligned}
&\quad\int_{P_\rho(x_0,t_0)}
(1+|F_A|)^{(\beta+2)(1+\frac2n)}\,dv\,dt
\\
&\leq
C\left[
\sup_{t_0-R^2\leq t\leq t_0}
\int_{B_R(x_0)}
(1+|F_A|)^{\beta+2}\phi^2\,dv
\right]^{\frac2n}
\\
&\quad\cdot
\int_{P_R(x_0,t_0)}
\left[
\left|
\nabla\left(
(1+|F_A|)^{\frac{\beta+2}{2}}\phi
\right)
\right|^2
+(1+|F_A|)^{\beta+2}\phi^2
\right]\,dv\,dt.
\end{aligned}
\end{equation}
Combining \eqref{2.20} and \eqref{2.21}, we obtain
\begin{equation}\label{2.22}
\begin{aligned}
&\int_{P_\rho(x_0,t_0)}
(1+|F_A|)^{(\beta+2)(1+\frac2n)}\,dv\,dt
\\
&\leq C\Bigg[
(\beta+2)
\int_{P_R(x_0,t_0)}
(1+|F_A|)^{\beta+3}\phi^2\,dv\,dt
\\
&\qquad+
\int_{P_R(x_0,t_0)}
(1+|F_A|)^{\beta+2}
\left(
|\phi\partial_t\phi|+|\nabla\phi|^2
\right)\,dv\,dt
\Bigg]^{1+\frac2n}.
\end{aligned}
\end{equation}
Now, we can perform Moser's iteration. Choose $R=\frac 12 (1+2^{-k})$, $\rho =R_{k+1}=\frac 12   (1+2^{-1-k})$ and $\beta=\left(1+\frac{2}{n}\right)^k+\frac{n-4}{2}$ for any integers $k\geq 0$. From (2.22), we get
\begin{equation}\label{eq:lemma5-iteration}
\begin{aligned}
&\quad \int_{P_{\frac{1+2^{-k-1}}2}(x_0,t_0)}
(1+|F_A|)^{
\left(1+\frac2n\right)^{k+1}+\frac{n+2}{2}}
\,dv\,dt
\\
&\leq
C4^{k(1+\frac2n)}
\left[
\int_{P_{\frac{1+2^{-k}}2}(x_0,t_0)}
(1+|F_A|)^{
\left(1+\frac2n\right)^k+\frac{n+2}{2}}
\,dv\,dt
\right]^{1+\frac2n}.
\end{aligned}
\end{equation}
This implies that
\begin{equation}
\begin{aligned}
&\left[
\int_{P_{\frac{1+2^{-k}}2}(x_0,t_0)}
(1+|F_A|)^{
\left(1+\frac2n\right)^k+\frac{n+2}{2}}
\,dv\,dt
\right]^{\left(1+\frac2n\right)^{-k}}
\\
&\leq
C^{\sum_{j=0}^{k-1}
\left(1+\frac2n\right)^{-j-1}}
4^{\sum_{j=0}^{k-1}
j\left(1+\frac2n\right)^{-j}}
\int_{P_1(x_0,t_0)}
(1+|F_A|)^{\frac n2+2}\,dv\,dt
\\
&\leq
C\int_{P_1(x_0,t_0)}
(1+|F_A|)^{\frac n2+2}\,dv\,dt.
\end{aligned}
\end{equation}
Letting $k\rightarrow \infty$ and applying Lemma \ref{Lemma  4} we get
\begin{equation}
\begin{aligned}
|F_A(x_0,t_0)|
&\leq
C\left(
1+\int_{P_1(x_0,t_0)}
|F_A|^{\frac n2+2}\,dv\,dt
\right)
\\
&\leq
C\left(
1+\int_{P_2(x_0,t_0)}
|F_A|^2\,dv\,dt
\right).
\end{aligned}
\end{equation}
Restoring the scaling, we get
\begin{equation}
|F_A(x_0,t_0)|^2
\leq
C\left(
\frac1{R_0^4}
+
\frac1{R_0^{n+2}}
\int_{P_{2R_0}(x_0,t_0)}
|F_A|^2\,dv\,dt
\right).
\end{equation}
This proves the claim.
\end{proof}

Under the local bound  of $F_A$, we apply an idea of Hamilton  \cite{Ha} on the Ricci flow to establish local estimates of higher order derivatives:
\begin{lem}  \label{Lemma 6}   Let $A(t)$ be a smooth solution to the Yang-Mills flow (\ref{YMF}) in $M\times
 [0,T]$ with initial connection $A(0)=A_0$.   Assume that there is a sufficiently small constant $\varepsilon>0$ such that
\begin{eqnarray}\label{A-G}
\sup_{ t_0-4R_0^2\leq t\leq t_0}  \int_{B_{2R_0}(x_0)} |F_A(x,t)|^{n/2}dv \leq \varepsilon
\end{eqnarray}
 for any $x_0\in M$ and $t_0>4R_0^2$. Then   for any integer $l\geq 1$, we have
 \begin{equation}\label{2.28'}
\sup_{t_0-R_l^2\leq t\leq t_0}
\int_{B_{R_l}(x_0)}
\sum_{i=1}^{l}
|\nabla_A^iF_A|^{\frac{2l}{i}}\,dv
+
\int_{P_{R_l}(x_0,t_0)}
|\nabla_A^{l+1}F_A|^2\,dv\,dt
\leq
C_l\sum_{i=1}^{l}
R_0^{\,n-2l-\frac{4l}{i}},
\end{equation}
   for some constant  $C_l > 0$ independent of  $A$ and $R_0$, where  $R_l:=\frac 12 R_0+\frac {R_0}{2^{l+1}}$ for  each $l\geq 1$.
\end{lem}
\begin{proof}
Without loss of generality, we assume that $R_0=1$ by  re-scaling and $t_0=0$. By Lemma \ref{Lemma  5}, we have
\begin{equation}
\sup_{M\times[-2,0]}|F_A|\leq C.
\end{equation}
We will prove this lemma by induction. From Lemma \ref{Lemma 1} and Lemma \ref{Lemma 2}, we have the induction basis
\begin{equation}\label{eq:energy-base}
\sup_{-1\leq t\leq0}
\int_{B_1(x_0)}|F_A|^2\,dv
+
\int_{P_1(x_0,0)}
|\nabla_AF_A|^2\,dv\,dt
\leq C.
\end{equation}
 Recall the well-known Ricci identity that for all $\psi\in \Omega^{k}(Ad E)$, we have
\begin{equation}\label{2.30}
\nabla_{A, e_i}\nabla_{A, e_j}  \psi=  \nabla_{A, e_j}\nabla_{A, e_i}
 \psi+ ( R_M+  F_A)\#   \psi.
\end{equation}
Combining \eqref{W2.1} and \eqref{2.30}, we obtain
\begin{equation}\label{2.32}
\begin{aligned}
\frac{\partial \nabla_A^lF_A}{\partial t}
={}&
-\nabla_A^*\nabla_A(\nabla_A^lF_A)+
\sum_{\substack{i+j=l\\i,j\geq0}}
\nabla_A^iF_A\#\nabla_A^jF_A+
\sum_{\substack{i+j=l\\i,j\geq0}}
\nabla^iR_M\#\nabla_A^jF_A.
\end{aligned}
\end{equation}
Assuming the $l-1$ case hold, we have
\begin{equation}\label{2.33}
\int_{P_{R_{l-1}}(x_0,0)}
|\nabla_A^lF_A|^2\,dv\,dt
\leq C_{l-1}.
\end{equation}
Let $\phi(x,t)$ be a cutoff function supported in $B_{R_{l-1}}(x_0)$ satisfying
\begin{equation}
\begin{gathered}
0\leq\phi\leq1,
\qquad
\phi=1\quad\text{on }P_{R_l}(x_0,0),
\\
\phi(\cdot,-R_{l-1}^2)=0,
\qquad
|\nabla\phi|^2+|\partial_t\phi|\leq C_l.
\end{gathered}
\end{equation}
Using integration by parts and Young's inequality, we have the interpolation inequality
\begin{equation}\label{2.35}
\begin{aligned}
&\int_M
|\nabla_A^iF_A|^{\frac{2l}{i}}\phi^{2l}\,dv
\\
&\leq
C_l\|F_A\|_{L^\infty(M)}^{\frac{2l}{i}-2}
\int_M|\nabla_A^lF_A|^2\phi^{2l}\,dv+
C_l\|F_A\|_{L^\infty(M)}^{\frac{2l}{i}}
\int_M|\nabla\phi|^{2l}\,dv
\end{aligned}
\end{equation} for each $i$ with $1\leq i\leq l$.
Combining this with the bound obtained in Lemma \ref{Lemma  5}, we obtain
\begin{equation}\label{2.36}
\int_M
|\nabla_A^iF_A|^{\frac{2l}{i}}\phi^{2l}\,dv
\leq
C_l\left(
1+\int_M|\nabla_A^lF_A|^2\phi^{2l}\,dv
\right).
\end{equation}
For $i,j\geq 1$ with $i+j=l$, the generalized Holder's inequality and \eqref{2.36} yields
\begin{equation}
\begin{aligned}
&\int_M
|\nabla_A^iF_A|\,
|\nabla_A^jF_A|\,
|\nabla_A^lF_A|\phi^{2l}\,dv
\\
&\leq
\left(
\int_M|\nabla_A^iF_A|^{\frac{2l}{i}}
\phi^{2l}\,dv
\right)^{\frac{i}{2l}}
\left(
\int_M|\nabla_A^jF_A|^{\frac{2l}{j}}
\phi^{2l}\,dv
\right)^{\frac{j}{2l}}
\left(
\int_M|\nabla_A^lF_A|^2\phi^{2l}\,dv
\right)^{\frac12}
\\
&\leq
C_l\left(
1+\int_M|\nabla_A^lF_A|^2\phi^{2l}\,dv
\right).
\end{aligned}
\end{equation}
Multiplying \eqref{2.32} by $\nabla^l_A F_A \phi^{2l}$ and integrating by parts yields
\begin{equation}
\begin{aligned}
&\frac{d}{dt}
\int_M|\nabla_A^lF_A|^2\phi^{2l}\,dv
+
\int_M|\nabla_A^{l+1}F_A|^2\phi^{2l}\,dv
\\
&\leq
C_l\left(
1+\int_M|\nabla_A^lF_A|^2\phi^{2l}\,dv
\right)+
C_l\int_M|\nabla_A^lF_A|^2
\left(
\phi^{2l-2}|\nabla\phi|^2
+\phi^{2l-1}|\partial_t\phi|
\right)\,dv
\\
&\leq
C_l\left(
1+
\int_{B_{R_{l-1}}(x_0)}
|\nabla_A^lF_A|^2\,dv
\right).
\end{aligned}
\end{equation}
Intgerating in time and applying \eqref{2.33}, we obtain
\begin{equation}
\begin{aligned}
&\sup_{-R_{l-1}^2\leq t\leq0}
\int_M|\nabla_A^lF_A|^2\phi^{2l}\,dv+
\int_{-R_{l-1}^2}^{0}
\int_M|\nabla_A^{l+1}F_A|^2\phi^{2l}\,dv\,dt
\\
&\leq
C_l\left(
1+
\int_{P_{R_{l-1}}(x_0,0)}
|\nabla_A^lF_A|^2\,dv\,dt
\right)
\leq C_l.
\end{aligned}
\end{equation}
Since $\phi=1\quad\text{on }P_{R_l}(x_0,0)$, we get
\begin{equation}
\begin{aligned}
&\sup_{-R_l^2\leq t\leq0}
\int_{B_{R_l}(x_0)}
\sum_{i=1}^{l}
|\nabla_A^iF_A|^{\frac{2l}{i}}\,dv+
\int_{P_{R_l}(x_0,0)}
|\nabla_A^{l+1}F_A|^2\,dv\,dt
\leq C_l.
\end{aligned}
\end{equation}
Restoring the original scaling, we have
\begin{equation}
\begin{aligned}
&\sup_{t_0-R_l^2\leq t\leq t_0}
\int_{B_{R_l}(x_0)}
|\nabla_A^iF_A|^{\frac{2l}{i}}\,dv
\leq C_lR_0^{\,n-2l-\frac{4l}{i}},
\\
&\int_{P_{R_l}(x_0,t_0)}
|\nabla_A^{l+1}F_A|^2\,dv\,dt
\leq C_lR_0^{\,n-2l-4}.
\end{aligned}
\end{equation}
Summing over $i$, we prove  the desired inquality.
 \end{proof}

\section{Parabolic Formulation of  the Gauge Fixing Theorem}

We  apply  Uhlenbeck's results \cite {U2} to obtain a parabolic formulation of  the gauge fixing theorem. Firstly, we recall  Uhlenbeck's gauge fixing theorem \cite {U2}.

\begin{thm}\label{U-G}
   Let $D_A=d+A \in W^{1,\frac{n}{2}}$ be a connection  in the  unit ball $B^n\subset \R^n$  satisfying
 \[\int_{B^n} |F_A|^{n/2}dv \leq \varepsilon   \] for some sufficiently small $\varepsilon >0$.
 Then there exists a  gauge
transformation $S\in G$ such that $D_A$ is gauge-equivalent to
$D_{a}:=S^{-1}\circ D_A \circ S$   with $a=S^*(A)=S^{-1} dS+S^{-1}AS$
 satisfying
 \[d^*  a =0  \mbox{ in }B^n,  \quad  a \cdot \nu=0  \mbox { on } \partial  B^n,
 \]
  where $\nu$ is the outer normal of $\partial B^n$,  and
\begin{equation*}\label{U2}\int_{B^n}  |a|^{n/2}+|\nabla a |^{n/2}\,dv\leq C \int_{B^n} |F_{a}|^{n/2}\,dv. \end{equation*}
 \end{thm}

The following is Lemma 2.7  of \cite {U2}, which is one of the key lemmas supporting the gauge fixing theorem.

\begin{lem}\label {Uh1}  For some $p>\frac n 2$, assume that $A\in W^{1,p}(B^n; \mathfrak{g})$  is a  given connection satisfying $d^*A=0$ in $B^n$ and
\[\|A\|_{L^{n}(B^n)}\leq k(n) \] with a sufficiently small $k(n)>0$.
Let $\lambda\in W^{1,p}(B^n; \mathfrak{g})$ be a one form satisfying $\lambda\cdot \nu=0$ on $\partial B^n$ and
\[\|\lambda\|_{W^{1,p}(B^n)}\leq \varepsilon \] for a sufficiently small constant $\varepsilon>0$.  Then
 there is a gauge transformation  $S=e^{u}\in W^{2,p}(B^n; G)$  such that $u\in \mathfrak{g}$ solves
\begin{equation}\label{PT}
d^*S^*(A+\lambda)=d^*(S^{-1}dS +S ^{-1}(A +\lambda   ) S ) =0  \end{equation}
 in $B^n$ with $\int_{B^n} u\,dx=0$ and $\partial_{\nu} u=0$ on $\partial B^n$.
\end{lem}

To obtain a parabolic formulation of the gauge fixing theorem,  we will also require two lemmas.
\begin{lem}\label{Lemma 2.8} For each $u(x, t)\in \mathfrak{g}$, set $s(x, t)=e^{-u(x, t)}\circ \frac {de^{u(x, t)} }{dt}$.
If $|u(t)|$ is bounded, then
\begin{equation}\label{a}
  \abs{\nabla s}\leq C\abs{\nabla  \frac {d u}{dt}} +C \abs{\nabla u}\abs{ \frac {d u}{dt}}.
\end{equation}
Moreover, if $|u(t)|$ is sufficiently small, then
  \begin{equation}\label{a.1}
  \abs{\frac {\partial \nabla u}{ \partial t}}\leq C\abs{\nabla s(t)} + C \abs{\nabla u}\,\abs{\frac {\partial u}{\partial t}}.
  \end{equation}
  \end{lem}
\begin{proof}

Note that
\[e^{u}=I + u + \frac {u^2}{2!} +\cdots + \frac {u^k}{k!}+\cdots .\]
Then
 \begin{equation}\label{a.2}  \frac {de^{u (t)} }{dt}=\frac {d u }{dt}+ \frac {\frac {d u}{dt} u(t)+ u(t)\frac {d u}{dt}
}{2!} +\cdots + \frac {\frac {d u}{dt}u^{k-1}+ \cdots +u^{k-1} \frac {d u}{dt}}{k!}+\cdots ,    \end{equation}
which implies
\[  \left |\frac {de^{u} }{dt}\right |\leq  \left |\frac {d u}{dt}\right | + \left |\frac {d u}{dt}\right | (e^{|u|}-1).\]
Similarly, we have
\[  \left |\nabla e^{-u}  \right |\leq \left |\nabla u\right | + \left |\nabla  u\right |   (e^{|u|}-1).\]
Using \eqref{a.2}, we find
\begin{eqnarray*}  \nabla \frac {de^{u} }{dt}  &=&
\nabla \frac {d u}{dt}+
\frac {\nabla \frac {d u}{dt} u + u \nabla \frac {d u}{dt}
}{2!}+\cdots + \frac {\nabla\frac {d u}{dt}u^{k-1} + \cdots +u^{k-1} \nabla\frac {d u}{dt}}{k!} +\cdots \\
&& +  \frac { \frac {d u}{dt}\nabla u + \nabla u   \frac {d u}{dt}
}{2!} +\cdots + \frac {\frac {d u}{dt} \nabla u^{k-1}+ \cdots +\nabla u^{k-1} \frac {d u}{dt}}{k!}+\cdots.   \end{eqnarray*}
For $k\geq 2$, we have
\[|\nabla u^{k-1}|=|\nabla uu^{k-2} +\cdots +u^{k-2}\nabla u|\leq (k-1)|\nabla u| |u|^{k-2}.\]
Then
 \begin{eqnarray*}   \left |\nabla \frac {de^{u } }{dt} \right |
 &\leq & \left |\nabla \frac {d u }{dt}\right |+ \left |\nabla \frac {d u }{dt}\right | \(e^{|u |}-1)+ |\frac {d u }{dt}| \,| \nabla u_   |(1+ |u |
+\cdots + \frac {|u |^{(k-2)}}{(k-2)!}+\cdots \) \\
&\leq &  \left |\nabla \frac {d u}{dt}\right |+ \left |\nabla \frac {d u}{dt}\right | (e^{|u|}-1)+ |\frac {d u}{dt}| \,| \nabla u |e^{|u|}.\end{eqnarray*}
Subsequently, \eqref{a} follows from the fact that
\[\nabla s (t)=\nabla e^{-u (t)}\circ \frac {de^{u (t)} }{dt}+e^{-u (t)}\circ \nabla\frac {de^{u (t)} }{dt}.\]

Using \eqref{a.2}, we get
\begin{equation}
\begin{aligned}
s
&= \left(I-u+\frac{u^2}{2!}-\cdots\right)
   \left(
     \frac{du}{dt}
     +\frac{1}{2!}\left(\frac{du}{dt}u+u\frac{du}{dt}\right)
     +\cdots
   \right) \\
&= \frac{du}{dt}
   -\frac{1}{2}\Bigl[u,\frac{du}{dt}\Bigr]
   +\frac{1}{6}\Biggl[u,\Bigl[u,\frac{du}{dt}\Bigr]\Biggr]
   -\cdots .
\end{aligned}
\end{equation}
This implies that
\begin{equation}\label{3.6}
\begin{aligned}
\nabla s
&= \nabla\!\left(\frac{du}{dt}\right)
   -\frac{1}{2}\Biggl(
     \left[\nabla u,\frac{du}{dt}\right]
     +\left[u,\nabla\!\left(\frac{du}{dt}\right)\right]
   \Biggr) \\
&\quad
   +\frac{1}{6}\Biggl(
     \left[\nabla u,\left[u,\frac{du}{dt}\right]\right]
     +\left[u,\left[\nabla u,\frac{du}{dt}\right]\right]
     +\left[u,\left[u,\nabla\!\left(\frac{du}{dt}\right)\right]\right]
   \Biggr)
   +\cdots .
\end{aligned}
\end{equation}
For $|u|$ sufficiently small, \eqref{3.6} implies \eqref{a.1}.
\end{proof}

\begin{lem}\label{LP} Let $f \in W^{1,p}(S^{n-1}\times [0,1])$ be a given function with $p>1$. Assume that $v=v(\theta ,r)$ is a solution of the heat equation on $S^{n-1}\times [0,1]$ satisfying
\begin {equation*}
 \partial_r v=\Delta_{S^{n-1}} v +f
\end{equation*} with $v(\theta, 1)=0$ on $S^{n-1}$.
Then we have
 \begin{equation}\label{et1}
 \|  \varphi v\|_{W^{1,p}(S^{n-1}\times [0,1])}\leq C\|f\|_{L^p(S^{n-1}\times [0,1])}
\end{equation} and
\begin{equation}\label{et2}
 \| \varphi v\|_{W^{2,p}(S^{n-1}\times [0,1])}\leq C\| f\|_{W^{1,p}(S^{n-1}\times [0,1])},
\end{equation}
where  $\varphi (r)$  is a smooth cut-off function in $[0,1]$ with $\varphi (r)=1$ near $1$ and $\varphi(r)=0$ for $[0,\eta]$ and some small constant $\eta >0$.
\end{lem}
\begin{proof} This lemma was mentioned by  Uhlenbeck \cite {U2} without details, so for completeness we give a proof here. Note that
\begin {equation}\label{et3}
 \partial_r (\varphi v)= \Delta_{S^{n-1}} (\varphi v )+\varphi f+\partial_r \varphi v.
\end{equation}
Following from standard $L^p$-estimates of parabolic equations (e.g. \cite{LSU}), we have
\begin{eqnarray*}
 &&\int_{S^{n-1}\times [0,1]}  |\partial_r (\varphi v)|^p + |\nabla^2_{\theta} (\varphi v)|^p + |\nabla_{\theta} (\varphi v)|^p d\theta dr\\
 &\leq& C \int_{S^{n-1}\times [0,1]}|\varphi f+\partial_r \varphi v|^pd\theta dr\\
 &\leq& C\int_{S^{n-1}\times [0,1]}| f|^p d\theta dr + C\int_{S^{n-1}\times [\eta,1]} |  v|^pd\theta dr.
\end{eqnarray*}
Then we claim that there exists constant $C > 0$ such that
\begin{eqnarray*}
  \int_{S^{n-1}\times [\eta,1]} |  v|^pd\theta dr \leq C\int_{S^{n-1}\times [0,1]}| f|^p d\theta dr.
\end{eqnarray*}
If not, then there exists a sequence of $k$ and solutions $v_k$ of the heat equation
\begin {equation*}
 \partial_r v_k=\Delta_{S^{n-1}} v_k +f_k
\end{equation*} with $v_k(\theta, 1)=0$ on $S^{n-1}$ such that
\begin {equation*}
 \int_{S^{n-1}\times [\eta,1]} |  v_k|^pd\theta dr \geq k\int_{S^{n-1}\times [0,1]}| f_k|^p d\theta dr.
\end{equation*}
Let $$\tilde v_k :=\frac {v_k}{\|v_n\|_{L^p(S^{n-1}\times [\eta,1])}}, \quad \tilde f_k :=\frac {f_k} {\|v_k\|_{L^p(S^{n-1}\times [\eta,1])}}.$$
which implies that $\|\tilde v_k\|_{L^p(S^{n-1}\times [\eta,1])}=1$ and $\tilde v_k$ satisfies
\begin {equation*}
 \partial_r \tilde v_k=\Delta_{S^{n-1}} \tilde v_k +\tilde f_k
\end{equation*} with $\tilde v_k(\theta, 1)=0$ on $S^{n-1}$. Noting that $\|\tilde f_k\|_{L^p(S^{n-1}\times [\eta,1])}\leq 1/k$,  there exists a function $v_{\infty}$ such that as $k\to\infty$, $\partial_r \tilde v_k$  converges to $\partial_r \tilde v_{\infty}$ and
$\nabla^2_{\theta}\tilde v_k$ converges to $\nabla^2_{\theta} v_{\infty}$ weakly in $L^{p}(S^{n-1}\times [\eta,1])$. By the Sobolev compact imbedding theorem (e.g. \cite{LSU}), $\tilde v_k$  converges  to $v_{\infty}$ strongly in $L^{p}(S^{n-1}\times [\eta,1])$, subsequently
\[\|v_{\infty}\|_{L^p(S^{n-1}\times [\eta,1])}=1.\] Furthermore,  $v_{\infty}$ also satisfies
\begin {equation*}
 \partial_r \tilde v_{\infty}=\Delta_{S^{n-1}} \tilde v_{\infty}
\end{equation*} with $\tilde v_{\infty}(\theta, 1)=0$ on $S^{n-1}$. Following from the backward uniqueness of the heat equation,  $\tilde v_{\infty}$ must be zero in $S^{n-1}\times [\eta,1]$.
This contradicts the fact that  $\|v_{\infty}\|_{L^p(S^{n-1}\times [\eta,1])}=1$. Therefore, the result follows, and (\ref{et1}) holds. Similarly, (\ref{et2}) can be demonstrated by differentiating (\ref{et3}) with respect to $r$.

\end{proof}

We can now provide a proof of the parabolic formulation of the gauge fixing theorem specified in Theorem \ref {Theorem 1.1}.

\begin{proof}[Proof of Theorem   \ref{Theorem 1.1}]

Without loss of generality assume that  $U=B^n$  and  $D_{A}=d+A$ in $U$. Using (\ref {small-condition}) at $t=0$,  it follows from   Uhlenbeck's
gauge fixing theorem in \cite {U2} that there exists a smooth gauge
transformation  $S_0=S(0)$ such that $D_{A(0)}$ is gauge-equivalent to a connection
$D_{a(0)}=S_0^*(D_{A(0)})=d+a(0)$ satisfying
\[ d^* a (0)=0\quad \mbox{in }U,\quad a(0)\cdot \nu=0\mbox { on }\partial  U\]
and
\[\int_{U}{|a (0)|^p}+|\nabla a(0)  |^p\,dv\leq
C(p)\int_{U} |F_{a(0)}|^p \,dv
\] for any $p\in [ \frac n2,n)$.

Note that $A$ is smooth in $M\times [0, t_1]$.  For each $p\in (n/2,n)$ and   a small $\varepsilon>0$
 there exists a  constant $\delta >0$, depending on the connection $A$ and $t_1$, such that for all $t, t'\in [0, t_1]$ with $|t-t'|\leq \delta$,
we have
\begin{eqnarray}\quad\label{condition}\int_{\bar U}|\nabla (A(t)- A(t'))|^p+|A(t)- A(t')|^p\,dv\leq \varepsilon^p.
 \end{eqnarray}

Now, we follow the procedure of Uhlenbeck \cite {U2} by fixing a Coulomb gauge for $t\in [0, \delta]$.
Through the gauge transformation $S_0$, the new connection $D_{\tilde A(t)}=S^*(0)(D_{A(t)})=d+\tilde A(t)$,
with $\tilde A (t)=S_0^{-1}d S_0+ S_0^{-1}A(t)S_0$, is a smooth solution the Yang-Mills flow in $\bar U\times [0,t_1]$ also.   Although  $\tilde A(0)=a(0)$ satisfies the boundary condition of $a(0)\cdot\nu =0$ on $\partial U$,
$\tilde A(t)$ may not satisfy the boundary condition of $\tilde A (t)\cdot\nu =0$ on $\partial U$ for $0<t\leq t_1$, so we cannot  apply Lemma   \ref{Uh1} to fix a Coulomb gauge for  $\tilde A(t)$  for $t\in [0, \delta]$. In order to rectify this issue of $\tilde A (t)$, it follows from Lemma 2.6 of \cite {U2}   that   there
exists gauge transformations $e^{u_1(t)}\in G$ such that
\[ D_{a_1(t)}=(e^{u_1(t)})^* (D_{\tilde A(t)})=e^{-u_1(t)}\circ (d+\tilde A(t))\circ e^{u_1(t)},\]
where $a_1(t):=\tilde A(0) +  \lambda (t)$ and
 \begin{equation} \label{(3.6)}\lambda (t) =-\tilde A(0)+e^{-u_1(t)}de^{u_1(t)}+ e^{-u_1(t)}(\tilde A(t))e^{u_1(t)}.
 \end{equation}
In fact,   $u_1(t)=\varphi (r) \tilde  v( t)$ is chosen with a smooth cut-off function $\varphi (r)$  in $[0,1]$ with $\varphi (r)=1$ near $1$ and $\varphi(r)=0$ for $[0,\eta]$ with some small constant $\eta >0$ and $\tilde v$ is the solution of
 \begin{equation}\label {P} \left(\frac {\partial }{\partial r} -\Delta_{S^{n-1}}\right)\tilde v= x\cdot (\tilde
A(0)-\tilde A(t))\quad \mbox{for $(r,\theta)\in [0,1]\times S^{n-1}$}\end{equation}
 with   $\tilde v(1,\theta)=0$ for all $\theta\in S^{n-1}$.
 Thus, we have
$u_1(t)=0$ and  $de^{u_1(t)}=d u_1(t)$ on $\partial U$ for all $t\in [0, \delta]$. Then using (\ref{(3.6)})-(\ref  {P}), we obtain
\begin{equation*}  \lambda (t)\cdot \nu = (d u_1(t)+\tilde
A(t)-\tilde A(0))\cdot \nu = \frac {\partial }{\partial r}\tilde v- x\cdot (\tilde
A(t)-\tilde A(0))=0 \quad \mbox{on $\partial
U$}.\end{equation*} Therefore, the  connection $a_1(t)$ satisfies the required boundary condition $a_1(t)\cdot \nu =0$ on $\partial U$.

Noting that  $\tilde A (t)=S_0^{-1}d S_0+ S_0^{-1}A(t)S_0$,   we have
 \[\tilde A (t)- \tilde A(0)= S_0^{-1}(A(t)-A(0))S_0.\]
 Using (\ref {condition}),  we determine
\begin{eqnarray}\quad\label{condition1} \| \tilde A (t)-\tilde A(0) \|_{W^{1,p}(U)}=\| A (t)-A(0) \|_{W^{1,p}(U)}\leq  \varepsilon
 \end{eqnarray} for any $t\in[0,  \delta]$.

Applying Lemma \ref {LP} and the Sobolev inequality to \eqref{P}, we have
\begin{equation}\label{Lp1}
\left (\int_U |\nabla u_1|^q(t)\,dv \right )^{1/q}\leq C\|\tilde A (t)-\tilde A(0)\|_{L^{q}(U)}\leq C \|\tilde A (t)-\tilde A(0)\|_{W^{1,p}(U)}
\end{equation} for $q=\frac {np}{n-p}>n$. From the Sobolev imbedding theorem and \eqref{condition}, $|u_1(t)|$ can be made sufficiently small for $t\in[0, \delta]$.
Furthermore, differentiating equation (\ref{P}) in $t$ yields
\[\frac {\partial u_1(t)}{\partial t} =\varphi \frac {\partial \tilde  v(t)}{\partial t} = -\varphi (\frac {\partial }{\partial r}
-\Delta_{S^{n-1}})^{-1} (x\cdot \frac {\partial \tilde A}{\partial
t}). \]
 Applying the $L^p$-estimate in Lemma \ref {LP} again, we have
\[ \int_U |\nabla \frac {\partial u_1(t)}{\partial t} |^2\,dv
\leq C  \int_U|\frac {\partial A}{\partial t}|^2(\cdot ,t)\,dv
\leq C \int_U|\nabla_A F_A|^2(\cdot ,t)\,dv\] for any $t\in [0,\delta]$.

From \eqref{(3.6)}, we deduce that
\[D_{a_1(t)}= d+  a_1(t)=d+\tilde A(0) +  \lambda (t)=d+e^{-u_1(t)}de^{u_1(t)}+ e^{-u_1(t)}(\tilde A(t))e^{u_1(t)}\] satisfies
\begin{eqnarray}\label{good1}
    \pfrac{ a_1}{t}=-D_{a_1}^*F_{a_1}+D_{a_1}   s_1,
    \end{eqnarray}
where $s_1= S_1^{-1}\circ  \frac d {dt} S_1$ and $S_1(t)=e^{u_1(t)}$.

Following from Lemma \ref{Lemma 2.8}, we have
\[|\nabla s_1(t)|\leq C \left|\nabla \frac {\partial u_1}{dt}\right|+C \left|\nabla u_1\right|\,\left|\frac {\partial u_1}{dt}\right|\]
for all $t\in [0, \delta]$ and for a sufficiently small $\delta>0$. Using the Sobolev inequality and noting that $u_1(t)=0$ on $\partial U$, we obtain
\begin{eqnarray*}
\int_U |\nabla s_1(t) |^2\,dv&\leq &  C \int_U \left|\nabla \frac {\partial u_1}{dt}\right|^2\,dv + \left (\int_U|\nabla u_1|^n\,dv\right )^{2/n}\left(\int_U\left|\frac {\partial u_1}{dt}\right|^{\frac {2n}{n-2}}\,dv \right)^{\frac {n-2}n} \\
 &\leq&  C \int_U \left|\nabla \frac {\partial u_1}{dt}\right|^2\,dv\leq  C \int_U|\nabla_A F_A|^2\,dv
\end{eqnarray*}
since
$\int_U |\nabla u_1|^n\,dv$ is uniformly bounded for all $t\in [0,\delta]$ according to \eqref{Lp1}.

Now,  we fix a Coulomb gauge for $A(t)$ with $t\in [0, \delta]$, then applying the $L^p$-estimate in Lemma \ref {LP} again, we have
\[|u_1(t)|\leq C\|u_1(t)\|_{W^{2,p}(U)}\leq C\|\tilde A(t)-\tilde A(0)\|_{W^{1,p}(U)}\leq  C\varepsilon\]
for all $t\in [0,\delta ]$. We note that
\begin{equation*} \lambda (t) =e^{-u_1(t)}\tilde A(0)e^{u_1(t)}-\tilde A(0)+e^{-u_1(t)}de^{u_1(t)}+ e^{-u_1(t)}(\tilde A(t)-\tilde A(0))e^{u_1(t)}.
 \end{equation*}
from which we determine that
\[\|\lambda (t)\|_{W^{1,p}(U)}\leq  C\varepsilon.\]

Using Lemma  \ref {Uh1} with a  sufficiently small constant $\varepsilon>0$,
 there exists $u_2(t)\in W^{2,p}(U)$ with
 $\nabla u_2\cdot
\nu =0$ on $\partial U$ and $\int_U u_2(t)\,dv=0$ for all $t\in
[0, \delta]$
 such that the gauge transformation  $S_2(t)=e^{u_2(t)}\in W^{2,p}(U)$   solves
\begin{equation}\label{small1}
d^*a=d^*(S_2^{-1}dS_2+S_2^{-1}(\tilde A(0)+\lambda (t) ) S_2) =0 \quad \mbox{in } U \end{equation}
  with $a\cdot\nu=0$ on $\partial U$,
where $D_a=d+a=(S_1S_2)^*(D_{\tilde A})=S_2^*(S_1^*(D_{\tilde A}))$.  It was noted by Uhlenbeck, in Lemma 2.7 of \cite {U2}, that $u_2(t)$  smoothly depends on $\lambda(t)$, so we can choose the norm
   $\|\nabla u_2(t)\|_{W^{2,p}(U)}$ to be sufficiently small since  $\|\lambda (t)\|_{W^{1,p}(U)}$ is small. Following from the Sobolev inequality, $|u_2(t)|$ can be made sufficiently small for $t\in [0, \delta]$.
 We verify this directly, note that the equation (\ref {small1}) is equivalent to
\begin{equation}\label{u_2}
-d^*  du_2 =d^*[(e^{u_2})^{-1}d e^{u_2(t)}- du_2] + \nabla e^{u_2} \#(\tilde A(0)+ \lambda (t)) \# e^{u_2} +e^{-u_2}d^*\lambda (t)   e^{u_2}.
\end{equation}
 Applying the $L^p$-estimate in (\ref{u_2}) and then  using both the H\"older and the Sobolev inequalities,  we have
\begin{eqnarray*}\label{small3}
&&\|u_2(t)\|^p_{W^{2,p}(U)} \leq  C\|\lambda (t) \|^p_{W^{1,p}(U)}+C \int_U|\nabla u_2(t)|^p( |\tilde A(0)|^p+|\lambda (t)|^p)\,dv\\
&\leq & C\|\lambda (t) \|^p_{W^{1,p}(U)}+ C\left (\int_U |\nabla u_2|^{\frac {pn}{n-p}}\,dv\right )^{\frac {n-p}n} \left (\int_U|\tilde A(0)|^n+ |\lambda (t)|^n \,dv\right )^{\frac {p}n}.
\end{eqnarray*}
Note that $\|\tilde A(0)\|_{L^n(U)}\leq C\varepsilon\leq k(n)$ and $\|\lambda (t)\|_{W^{1,p}(U)}\leq C\varepsilon$. Then for a sufficiently small $\varepsilon >0$, we have
\begin{equation}\label{small4}
\|u_2(t)\|_{W^{2,p}(U)} \leq   C\|\lambda (t) \|_{W^{1,p}(U)}\leq C \varepsilon
\end{equation}
for $p> n/2$.

   From Lemma \ref{Lemma 2.8}, we have
\[C^{-1} \left|\frac {\partial u_2}{\partial t} \right|\leq |s_2|\leq C \left|\frac {\partial u_2}{ \partial t} \right|\] and
   \[\left|\nabla \frac {\partial u_2}{ \partial t}\right|\leq C|\nabla s_2(t)| +C |\nabla u_2|\,\left|\frac {\partial u_2}{ \partial t}\right|\]
   for  $t\in [0,\delta]$.
   Since $\int_Uu_2(t)\,dv=0$, we
have $\int_U \frac {\partial u_2}{ \partial t}\,dv=0$, and because $\int_U   |\nabla u_2(t)|^n\,dv$ can be chosen to be sufficiently small for $t\in [0,\delta]$, we have
 \begin{eqnarray*} &&\int_U\left|\nabla \frac {\partial u_2}{ \partial t}\right|^2\,dv\leq C\int_U |\nabla s_2(t)|^2+ |\nabla u_2|^2\,\left|\frac {\partial u_2}{\partial t}\right|^2\,dv\\
 &&\leq C\int_U |\nabla s_2(t)|^2\,dv + C\left (\int_U   |\nabla u_2|^n\,\right )^{2/n}\left (\int_U \left|\frac {\partial u_2}{\partial t}\right|^{\frac {2n}{n-2}} \,dv\right )^{\frac {n-2}n}  \\
 &&\leq C\int_U |\nabla s_2(t)|^2\,dv+\frac 1 2 \int_U\left|\nabla \frac {\partial u_2}{ \partial t}\right|^2\,dv.\end{eqnarray*}
It implies that
   \[\int_U | s_2(t)|^2\,dv\leq C \int_U \left|\frac {\partial u_2}{\partial t} \right|^2\,dv \leq C\int_U \left|\nabla \frac {\partial u_2}{ \partial t}\right|^2\,dv\leq C\int_U |\nabla s_2(t)|^2\,dv. \]
Let $s_2(t)=S_2^{-1}\circ \frac {d S_2}{dt}$ with $S_2=e^{u_2}$, then
$D_a=S_2^* (D_{a_1})$ satisfies (\ref{good}); i.e.
 \begin{eqnarray} \label{good2}
    \pfrac{a}{t}=-D_{a}^*F_{a}+D_{a} s \quad \mbox{in } U\times [0,\delta]
    \end{eqnarray}
   with $s= (S_1(t)S_2(t))^{-1}\circ \frac {d (S_1(t) S_2(t))}{dt}=  S_2^{-1}(t)s_1(t) S_2(t) +  S_2^{-1}(t)
   \circ \frac {d S_2}{dt}$.

Since $d^*a=0$ in $U$ and $a\cdot\nu=0$ on $\partial U$, it follows from Lemma 2.5 of \cite {U2} that for all $t\in [0,\delta_1]$
\begin{eqnarray*}
&& \int_{U}|a(\cdot, t)|^{n/2}+|\nabla a (\cdot, t)|^{n/2}\,dv \leq C\int_{U}|F_a(\cdot, t)|^{n/2}\,dv \leq C\varepsilon.
\end{eqnarray*}
Using the Sobolev inequality, we have
\[\|a\|_{L^n(U)}\leq C\|a(\cdot, t)\|_{W^{1,n/2}(U)}\leq C\varepsilon.\]
Recall that $s(t) =  S_2^{-1}(t)s_1(t) S_2(t) + s_2(t)$, such that
\[\int_U  \left
<d s, \pfrac{a}{t} \right >\,dv =\int_U\left <\frac {\partial s}{\partial x_k}, \frac{\partial a_k}{\partial t} \right >\,dv =\int_U  \left<
 s, \pfrac{d^*a}{t} \right >\,dv + \int_{\partial U } \left <  s ,  \partial_ta \cdot \nu  \right >ds=0. \]
Then applying the H\"older and  the Sobolev inequalities, we have
\begin{eqnarray}\label{key}
&&\int_U \left <D_{a} s, \pfrac{a}{t} \right >\,dv=\int_U  \left
<d s, \pfrac{a}{t} \right > +\left < [a, S_2^{-1}(t)s_1(t) S_2(t) + s_2(t) ], \pfrac{a}{t}  \right
>\,dv\\
&&\leq  C\left (\int_U |a|^n\,dv\right )^{2/n}
\left[\left(\int_U|s_1|^{\frac {2n}{n-2}}\,dv\right)^{\frac {(n-2)}n}+\left(\int_U|s_2|^{\frac {2n}{n-2}}\,dv\right)^{\frac {(n-2)}n}\right]\nonumber\\
&& \quad + \frac 14 \int_U | \pfrac{a}{t}|^2 \,dv \nonumber\\
&& \leq  C \varepsilon
\int_U |\nabla s_1 |^2+ |\nabla s_2 |^2 \,dv +\frac 14 \int_U | \pfrac{a}{t}|^2 \,dv.\nonumber
\end{eqnarray}
Using (\ref{good2}), we deduce
\begin{eqnarray*}
 \int_U  \left|D_{a} s-\pfrac{a}{t} \right|^2\,dv  \leq C \int_U |\nabla_a
    F_a|^2\,dv .
\end{eqnarray*}
Since
$s(t)=  S_2^{-1}(t)s_1(t) S_2(t) +  s_2(t)$ with $S_2=e^{u_2}$, we have
\[|\nabla s_2|\leq |D_a s|+|\nabla s_1|+ C|\nabla S_2| \, |s_1|+|a| (|s_1|+|s_2|).\]
Following from (\ref {small4}) we can make $\|\nabla S_2\|_{L^n(U)}$ sufficiently small.  Applying the H\"older and  the Sobolev inequalities, we determine
\begin{eqnarray*}
&&\int_U |\nabla s_2|^2\,dv \leq C\int_U |D_a s|^2+|\nabla s_1|^2+ |\nabla S_2|^2 \, |s_1|^2+|a|^2 (|s_1|^2+|s_2|^2)\,dv \\
&& \leq C\int_U |D_a s|^2+|\nabla s_1|^2\,dv + C \left (\int_U|\nabla S_2|^n\,dv \right )^{\frac 2n}\left (\int_U |s_1|^{\frac {2n}{n-2}}  \,dv \right )^{\frac {n-2}n}\\
&&+ C \left (\int_U|a|^n\,dv \right )^{\frac 2n}\left (\int_U (|s_1|^{\frac {2n}{n-2}}+|s_2|^{\frac {2n}{n-2}})\,dv  \right )^{\frac {n-2}n}\\
&&\leq C\int_U |D_a s|^2+|\nabla s_1|^2\,dv + C\varepsilon \int_U  |\nabla s_2 |^2 \,dv .
\end{eqnarray*}
 Choosing $\varepsilon > 0$ sufficiently small in (\ref{key}), we obtain
\begin{eqnarray} \label{point}
\int_U |s|^2+|D_{a} s|^2 +\left| \pfrac{a}{t}\right|^2 \,dv \leq C \int_U |\nabla_a
    F_a|^2 \,dv  \end{eqnarray}
    for any $t\in [0 , \delta]$. Then
 we have
    \[  \int_{0}^{\delta }\int_U |s|^2+|D_{a} s(t)|^2 +| \pfrac{a}{t}|^2 \,dv\,dt\leq C  \int_{0}^{\delta }\int_U |\nabla_a
    F_a|^2\,dv\,dt.\]
    Moreover, we have $d^*a(t)=0$ in $U$ and $a(t)\cdot\nu =0$ on $\partial U$ for $t\in [0, \delta]$.

    For the above choices of $\delta$, we assume that $\delta \leq t_1$. If  $\delta <t_1$, then we  repeat the above procedure starting  at $t_0=\delta$ instead of at $t_0=0$; i.e.
at $t_0=\delta$, there is  a gauge transformation $\tilde S=S(\delta)=e^{u(\delta)}$ such that $D_{\tilde A(t)}=d+\tilde A(t)= \tilde S^*(D_A)$ is a smooth in $\bar U$ such that at $t=\delta$
\[d^*  \tilde A(\delta )=0\quad \mbox{ in }U,\quad  \tilde A(\delta ) \cdot \nu =0\mbox { on } \partial U.\]
Since $\tilde S$ is a  smooth gauge transformation,  $\tilde S^*(D_A)$ is a smooth solution of the Yang-Mills flow (\ref{YMF}).
Then repeating the procedure again,   we obtain two more smooth solutions  $u_1(t)$ and  $u_2(t)$ on $[\delta, 2\delta]$ starting at $t=\delta$ with $\tilde u_2(\delta)=0$.
 More precisely, there exists gauge transformations $S_1(t)= e^{u_1(t)}$ and $\tilde S_2(t)=e^{u_2(t)}$   for any
$t\in [\delta,  2\delta]$, with   initial condition $\tilde u_1(\delta)=0$ and  $\tilde u_2(\delta)=0$,  and connection
\[D_{a(t)}=S(t)^*(D_{A(t)})=(e^{u_2(t)})^*\circ(e^{u_1(t)})^*\circ (\tilde S^*(D_{A(t)}))\]
for $t\in [\delta, 2\delta]$
 satisfying
the same equation (2.19) (or (2.22)) in $U\times [\delta, 2\delta]$ with initial connections $\tilde u_2(\delta )=0$ and $\tilde u_2(\delta)=0$.

Subsequently, for a given $t_1>0$, there are finitely many numbers $l$ depending on $\delta$ and $t_1$ such that
\[[0,t_1]\subset  [0, l\delta].\]
In conclusion, for any $t\in [0,t_1]$, there exists gauge transformations $S(t)$ such that connections $D_{a(t)}=S^*(t)(D_{A(t)}) $ satisfying equations (\ref{good})-(\ref{NU}) in $U\times [0,t_1]$. This completes the proof of Theorem \ref {Theorem 1.1}.

 \end{proof}

 \section{Compactness Theorem and  Existence of Weak Solutions}

We now proceed to prove Theorem \ref{Theorem 1.2}, but first we prove the following compactness theorem.
 \begin{thm}\label {Theorem 4.1} Let $D_{A^k}$ be a sequence of smooth   solutions  of the Yang-Mills
flow (\ref{YMF}) in $\bar U\times [0,t_1]$ with a uniform constant $t_1>0$, where $U=B_{r_0}(x_0)$ for some $r_0>0$.    Assume that there are uniform constants $\varepsilon > 0$ and   $C > 0$ dependent on $k$ such that
\begin{equation}\label{GG1}\sup_{0\leq t\leq t_1} \int_{U} |F_{A^k(x,t)}|^{n/2}dv \leq \varepsilon   \end{equation}
and
 \begin{equation}\label{GG2}  \int_{0}^{t_1 } \int_{U} |\nabla_{A^k(x,t)}
    F_{A^k(x,t)}|^2\,dv\,dt\leq C. \end{equation}
Then   the smooth
solutions $D_{A^k}$ converge smoothly, up to gauge transformation, to a connection $D_A$ in $U\times (0,t_1]$. Furthermore, $D_A$ is a smooth solution of the Yang-Mills flow in  $U\times (0,t_1]$.
\end{thm}

\begin{proof}Let $D_{A^k(t)}$ be  a sequence of  smooth solutions of the Yang-Mills flow in $U\times [0,t_1]$ satisfying
 (\ref{GG1})  and (\ref{GG2}).
Following from Theorem \ref {Theorem 1.1}, there exists gauge
transformations $S^{k}(t)=e^{u^k(t)}$ and  connections
$D_{a^k}=(S^k)^*(D_{A^k})=d+a^k$ such that
\[  d^*  a^k(t)=0 \quad \mbox {in } U,\quad  a^k(t)\cdot \nu=0 \mbox { on }\partial  U\]
and
 \[\int_{U} r_0^{-{n/2}}|a^k(t)|^{n/2}+|\nabla a^k(t)|^{n/2}\,dv\leq C \int_{U} |F_{a^k(t)}|^{n/2}\,dv\leq C\varepsilon\]
 for any $t\in [0, t_1]$. Furthermore,
 $D_{a^k}$ is a  solution of the equation
\begin{eqnarray}\label{EQ}
    \pfrac{a^k}{t}=-D_{a^k}^*F_{a^k}+D_{a^k} s^k
    \end{eqnarray}
   in $U\times [0,t_1]$, where \[s^k(t)=(S^k)^{-1}(t)\circ \frac d{dt} S^k(t),  \]
and
   \begin{equation}\label{NU1}  \int_{0}^{t_1 }\int_{U}\frac 1{r_0^2}| s^k|^2+ |D_{a^k} s^k|^2 +| \pfrac{a^k}{t}|^2 \,dv\,dt\leq C \int_{0}^{t_1 } \int_{U} |\nabla_{a^k}
    F_{a^k}|^2\,dv\,dt\leq C. \end{equation}

Letting $k\to\infty$, $(a^k, s^k)$ converges to $(a, s)$, which is a solution of
\begin{eqnarray}\label{YEF}
    \pfrac{a}{t}=-D_{a}^*F_{a}+D_{a} s
    \end{eqnarray}
   in $U\times [0,t_1]$.  Also, we have
\begin{equation}\label{BU}  d^*  a(t)=0\quad \mbox {in } U,\quad a(t)\cdot \nu=0 \mbox { on }\partial  U \end{equation}
and
\[\int_{U} |F_{a(t)}|^{n/2}\,dv\leq \liminf_{k\to 0} \int_{U} |F_{a_i^k(t)}|^{n/2}\,dv\leq \varepsilon.\]
Applying the  condition (\ref{BU}), we have
 \[\int_{U} r_0^{-{n/2}}|a(t)|^{n/2}+|\nabla a(t)|^{n/2}\,dx\leq C \int_{U} |F_{a(t)}|^{n/2}\,dv\leq C\varepsilon\]
for all $t\in [0,t_1]$ and
\begin{equation} \label{es}  \int_{0}^{t_1 }\int_{U}\frac 1{r_0^2}|s|^2+ |D_{a} s|^2 +| \pfrac{a}{t}|^2 + |\nabla_{a}
    F_{a}|^2\,dv\,dt\leq C. \end{equation}

Using  (\ref{YEF})-(\ref{BU}) and the identity $D_a^*D_a^*F_a=0$ (see \cite{St3}),  we
deduce
\begin{eqnarray}\label{as}
d^*ds&=&a\#D_{a}^*F_{a}+a\#\nabla s .
      \end{eqnarray}
Let $\phi$ be a cut-off function in $C_0^{\infty}(B_{\frac 34 r_0})$ with $\phi =1$ in $B_{r_0/2}$.  Then, multiplying (\ref{as}) with $\phi$, we obtain
\begin{eqnarray*}
d^*d(s\phi)=a\#D_{a}^*F_{a}\phi+a\#\nabla (s \phi)+a\#\nabla  \phi s +s d^*d\phi +d\phi \# ds .
      \end{eqnarray*}
Applying the $L^p$-estimates  to the above equation and using the H\"older and  Sobolev inequalities,  we have
 \begin{eqnarray*}
 &&\int_{B_{\frac 34 r_0}}|\nabla^{2}(s\phi)|^2\,dv\leq C\int_{B_{\frac 34 r_0}}| d^*d(s\phi)|^2\,dv\\
 &&\leq C\int_{B_{\frac 34 r_0}}|a|^2 (|D_{a}^*F_{a}\phi|^2+ |\nabla (s \phi)|^2)+|a|^2 |\nabla  \phi s|^2 +|s|^2 |\nabla^2 \phi|^2 +|d\phi|^2 |ds|^2 \,dv\\ &\leq&  C\left (\int_{B_{ r_0}}|a|^n\,dv\right )^{\frac 2 n}\left [ \left ( \int_{B_{\frac 34 r_0}}  |\nabla_a F_a|^{\frac {2n}{n-2}} dv\right )^{\frac {n-2} n}+ \left (\int_{B_{\frac 34 r_0}} |\nabla (\phi s)|^{\frac {2n}{n-2}}dv\right )^{\frac {n-2} n}\right  ]
 \\
 &+ &  C\left (\int_{B_{\frac 34 r_0}}|a|^n\,dv\right )^{\frac 2 n}\left (\int_{B_{\frac 34 r_0}} |s|^{\frac {2n}{n-2}}\,dv\right )^{\frac {n-2} n}+C\int_{B_{\frac 34 r_0}} |\nabla s|^2+|s|^2dv \\
 &\leq &  C \varepsilon \int_{B_{\frac 34 r_0}}    |\nabla^2 (\phi s)|^2 \,dv+C \int_{B_{r_0}}    |\nabla^2_a F_a|^2 +|F_a|^2+|\nabla s|^2+|s|^2\,dv. \end{eqnarray*}
For a sufficiently small $\varepsilon>0$, we obtain
\begin{eqnarray}\label{a12}
\int_{B_{\tfrac 34 r_0}} |\nabla^{2}s|^2 dv\leq C \int_{B_{\tfrac 34 r_0}} |\nabla^2_a F_a|^2  +|\nabla_a F_a|^2+|F_a|^2\,dv+C.  \end{eqnarray}

Using (\ref {YEF}), we have
\begin{eqnarray}\label{YEF1}
    \pfrac{a}{t}=\Delta a +a \# F_a +\nabla a\# a + D_as.
    \end{eqnarray}
Then  $\partial_ta, \nabla^2a\in
    L^{2}(B_{(1-\theta)r_0}\times [\theta,t_1] )$  for any $\theta >0$ (see, e.g.,  \cite{LSU}).
    Using $d^* a =0$ and the fact that $F_a$ is bounded inside $B_{(1-\theta)r_0}\times [\theta,t_1]$, we have
    \[\sum_{j=0}^{l+1} \int_{\frac 12 U} |\nabla^j  a|^2\,dv\leq C_l\sum_{j=1}^{l}\int_U |\nabla_{a}^{j} F_{a}|^2\,dv. \]

  It follows from the Uhlenbeck  gauge fixing theorem that  at each $t_0 >0$ there exists gauge transformations $S^k(t_0
    )$ such that  $\tilde A^k(t) := S^k(t_0)(A^k(t))$ satisfying the Yang-Mills flow (\ref{YMF}),
   $d^*  \tilde A^k(t_0)=0$ in $U$ and $\tilde A^k(t_0)\cdot \nu =0$ on $\partial U$. According to Lemma \ref{Lemma 5}, there exists a uniform constant $C(t_0) > 0$ such that
    \[|F_{\tilde A^k}(x, t_0)|\leq C(t_0).\]
    For each integer $l\geq 1$, it follows from Lemma \ref{Lemma 6} that
     \[ \int_{\frac 34 U}|F_{\tilde A^k}(x, t_0)|^2+\cdots +|\nabla^l F_{\tilde A^k}(x, t_0)|^2 \,dv\leq  C_l(t_0).\]
   Since $d^*  \tilde A^k(t_0)=0$, we have
    \[\sum_{j=0}^{l-1} \int_{\frac 12 U} |\nabla^j \tilde A^k(x,t_0)|^2\,dv\leq C\sum_{j=1}^{l}\int_{\frac 34U} |\nabla_{\tilde A^k}^{j} F_{\tilde A^k}|^2\,dv. \]
  Then, using the Yang-Mills flow (\ref{YMF}), we obtain
   \begin{eqnarray*}
   &&\sup_{0\leq t\leq t_1}\sum_{j=0}^l \int_{\frac 12 U} |\nabla^j \tilde A^k(x,t_1)|^2\,dv\\
   &\leq &\sum_{j=0}^l \int_{\frac 12 U} |\nabla^j  \tilde A^k(x,t_0)|^2\,dv
    + C\sum_{j=1}^{l-1}\int_{t_0}^{t_1}\int_{\frac 34U} |\nabla_{\tilde A^k}^{j-1} F_{\tilde A^k}|^2\,dv dt.
    \end{eqnarray*}
Therefore,
   $\tilde A^k$  converges to a smooth solution  $\tilde A$ of the Yang-Mills flow in $U\times (0,t_1]$, which implies that $a$ is gauge-equivalent to a smooth solution.
\end{proof}

As a consequence of Theorem \ref {Theorem 4.1}, we have the following local existence theorem.
\begin{thm}\label{Theorem 4.2} For $n\geq 4$,  let $D_{A_0} \in W^{1,\frac{n}{2}}(M)$  be a connection with   curvature $F_{A_0}\in L^{n/2}(M)$.
Assume that there exists a uniform constant $\varepsilon_0 >0$ such that
\begin{equation}  \label{S-C}
\sup_{x_0\in M}\int_{B_{R_0}(x_0)} |F_{A_0}|^{n/2}\,dv < \varepsilon_0
\end{equation}
for some $R_0>0$.
 Then there exists
a local weak solution of the Yang-Mills flow (\ref{YMF}) with initial connection $A_0$ in $M\times [0, t_1]$ for some $t_1\geq \  \frac {R_0^2} C$ and $C>1$. Furthermore, for each $t\in  (0,t_1]$, the solution $A(t)$ is gauge-equivalent to a smooth solution of the Yang-Mills flow.
  \end{thm}
\begin{proof}
There exists a sequence of  smooth connections $\{A^k(0)\}_{k=1}^{\infty}$, which converge strongly to $A_0 \in W^{1, \frac{n}{2}}(M)$. Since $A^k(0)$ is smooth, $D_{A^k(t)}=D_{\text{ref}}+A^k(t)$ is the unique  smooth solution  of the Yang-Mills flow  with initial connection $A^k(0)$  for some maximal existence time $T_k>0$.
 We claim that there exists a uniform constant $t_1>0$ such that $T_k\geq t_1$ for all $k\geq 1$.

Using \eqref{S-C}, for some small constant $\varepsilon >0$, there exists a uniform constant $R_0>0$ such that for any point $x_0\in M$,
\[\int_{B_{R_0}(x_0)} |F_{A^k_0}|^{n/2} \,dv\leq \frac {\varepsilon} 2.  \]
 Since $M$ is compact, there is a finite  cover of open balls $\{\frac 12 U_i\}_{i=1}^L$ of $M$ with $U_i=B_{R_0}(x_i)$, such that at each $x\in M$, there are at most a finite number $K$ of  balls intersecting.
Following from Lemma \ref{Lemma 3}  and a covering argument over $M$, we have
  \begin{equation*}  \int_{B_{R_0}(x_0)} |F_{A^k(t)}|^{n/2}  \,dv\leq \int_{B_{2R_0}(x_0)} |F_{A^k_0}|^{n/2}  \,dv +\frac {CKt}{R_0^2} \sup_{0\leq s\leq t} \max_{x\in M}\int_{B_{R_0}(x)} |F_{A^k(s)}|^{n/2}  \,dv  \end{equation*}
  which implies
\begin{equation*} \max_{x\in M} \int_{B_{R_0}(x)} |F_{A^k(t)}|^{n/2}  \,dv\leq \varepsilon \end{equation*}
for all $t\in [0,\frac { R_0^2}{2KC}]$. It follows that each smooth solution $A^k(t)$ of the Yang-Mills flow can be extended to the unform time $t_1 =  \frac {R_0^2}{2KC}>0$  such that
 \[\sup_{0\leq t\leq t_1}\int_{B_{R_0}(x_0)}|F_{A^k(t)}|^{n/2}\,dv\leq \varepsilon \]
and
 \begin{equation*} \int_{0}^{t_1 }\int_{M}   |\nabla_{A^k}
    F_{A^k}|^2\,dv\,dt\leq C. \end{equation*}

Applying Theorem \ref{Theorem 1.1} on each  $U_i$ we obtain gauge
transformations $S_i^{k}(t)=e^{u^k(t)}$ and connections
$D_{a^k}=(S_i^k)^*(D_{A^k})=d+a_i^k$ such that
\[  d^*  a_i^k=0 \quad \mbox {in } B_{r_0}(x_0),\quad a_i\cdot \nu=0 \mbox { on }\partial  U_i,\]
satisfying
 \[\int_{U_i} r_0^{-n/2}|a_i^k(t)|^{n/2}+|\nabla a_i^k(t)|^{n/2}\,dv\leq C \int_{U_i} |F_{a_i^k(t)}|^{n/2}\,dv\leq C\varepsilon\]
for  all $t\in [0, t_1]$,
and for which $D_{a_i^k}$ is a  solution of the equation
\begin{eqnarray}\label{4.5}
    \pfrac{a_i^k}{t}=-D_{a_i^k}^*F_{a_i^k}+D_{a_i^k} s_i^k
    \end{eqnarray}
   in $B_{r_0}(x_0)\times [0,t_1]$, where \[s_i^k(t)=(S_i^k)^{-1}(t)\circ \frac d{dt} S_i^k(t)   \] and
 \begin{equation*} \int_{0}^{t_1 }\int_{U_i}   (|\pfrac{a_i^k}{t}|^2 +|s^k_i|^2+|D_{a_i^k} s_i^k|^2)\,dv\,dt\leq \int_{0}^{t_1 }\int_{U_i}   |\nabla_{a^k}
    F_{a^k}|^2\,dv\,dt\leq C\int_M |F_{A_0}|^2\,dv. \end{equation*}

 For any $\delta>0$ and each $l \geq 1$, using Lemma \ref{Lemma 6}, there exists  constants $C(\delta, l)>0$ such that
\begin{equation*}
   \int_{\delta}^{t_1} \int_M|\nabla_{a^k} F_{a^k}|^2+\cdots +|\nabla^l_{a^k} F_{a^k}|^2\,dv\, dt\leq C(\delta, l).
\end{equation*}
Since $d^* a_i^k=0$ in $U_i$, we have $\Delta a_i^k= -d^*d a_i^k=d^*F_{a}+da\# a$. Under these conditions, we determine that
\begin{equation*} \int_{0}^{t_1 }\int_{\frac 1 2U_i}   |\nabla^{l+1} {a_i^k}|^2\,dv \,dt\leq C(\delta, l+1)
    . \end{equation*}

Then as $k\to\infty$, $a^{k}_i$ converges to $a_i$  for any $t\in [\delta, {t_1}]$ satisfying that $d^*a_i(t)=0$ in $U_i$ with $a_i(t)\cdot\nu =0$ on $\partial U_i$ for  $t\in [\delta , {t_1}]$,
 \[\int_{U_i} \frac 1{r_0^{n/2}}|a_i(t)|^{n/2}+|\nabla a_i(t)|^{n/2}\,dv\leq C \int_{U_i} |F_{a_i
 (t)}|^{n/2}\,dv\]
   and
  \begin{eqnarray}\label{4.8}
 &&\frac d{dt} a_i= - D_{a_i}^*F_{a_i}+ D_{a_i} s_i.
    \end{eqnarray}
Furthermore,  we have
    \[  \int_{0}^{t_1}\int_{U_i} |D_{a_i} s_i(t)|^2 +| \pfrac{a_i}{t}|^2 \,dv\,dt\leq C  \int_{0}^{t_1}\int_{U_i} |\nabla_{a_i}
    F_{a_i}|^2\,dv\,dt\]
    and
\begin{equation*}
   \int_{\delta}^{t_1} \int_{\frac 12 U_i}|\nabla^{l+1} a_i|^2+|\nabla^l_{a_i} F_{a_i}|^2\,dv\, dt\leq C(\delta, l)
\end{equation*}  for each $l\geq 1$.

In the local trivialization of $E_{U_i}$,  $d+ a^{k}_i$ can be regarded as a local representative of  the   connection $D_{a^k}$ of $E$ over $U_i$. In the overlap $\frac 12 U_i\cap \frac 12 U_j$ of two balls,  $a^k_i$ and $a^k_j$ can be identified as the equivalent by a gauge transformation  through $S_{ij;k}\in   G$ between  $E_{U_i}$ and $E_{U_j}$ (see Lemma 3.5 of \cite{U2}) such that
  \begin{eqnarray*}
&&a^k_j= S^{-1}_{ij; k}dS_{ij; k}+S^{-1}_{ij; k}a^{k}_iS_{ij; k},\\
&& d  a^k_j= dS^{-1}_{ij; k}\wedge dS_{ij; k}+S^{-1}_{ij; k}da^{k}_iS_{ij; k}+d S^{-1}_{ij; k}\wedge a^{k}_iS_{ij; k}-S^{-1}_{ij; k}a^{k}_i \wedge d S_{ij; k}.
  \nonumber
\end{eqnarray*}
Moreover, we have
\[ D_{a^{k}_j}=S_{ij; k}^*(D_{a^{k}_i})=S_{ij; k}^{-1}\circ D_{a^{k}_i} \circ S_{ij; k},\quad s^{k}_j=S_{ij; k}  s^{k}_i S_{ij;k}^{-1}+S_{ij;k}^{-1}\circ \frac d{dt} S_{ij;k}  \]
and
\[F_{a^{k}_j}=S_{ij;k}^{-1} F_{a^{k}_i} S_{ij;k},\quad D_{a_j^k}s_j^k =D_{a_j^k}\!\left(
S_{ij;k}^{-1}s_i^kS_{ij;k}
+S_{ij;k}^{-1}\frac{\partial S_{ij;k}}{\partial t}
\right). \]
From which we deduce
\begin{equation}
\begin{aligned}
\frac{\partial a_j^k}{\partial t}
&=S_{ij;k}^{-1}\frac{\partial a_i^k}{\partial t}S_{ij;k}
+D_{a_j^k}\!\left(
S_{ij;k}^{-1}\frac{\partial S_{ij;k}}{\partial t}
\right)\\
&=-S_{ij;k}^{-1}(D_{a_i^k}^{*}F_{a_i^k})S_{ij;k}
+S_{ij;k}^{-1}(D_{a_i^k}s_i^k)S_{ij;k}\\
&\qquad+D_{a_j^k}\!\left(
S_{ij;k}^{-1}\frac{\partial S_{ij;k}}{\partial t}
\right)\\
&=-D_{a_j^k}^{*}F_{a_j^k}
+D_{a_j^k}\!\left(
S_{ij;k}^{-1}s_i^kS_{ij;k}
+S_{ij;k}^{-1}\frac{\partial S_{ij;k}}{\partial t}
\right)\\
&=-D_{a_j^k}^{*}F_{a_j^k}+D_{a_j^k}s_j^k,
\end{aligned}
\end{equation}
 which shows that the equation (\ref  {4.5}) is defined globally on $M$. Furthermore, equation  (\ref{4.8}) is well-defined in $U_i\cap U_j$ since
  \begin{eqnarray*} &&\frac d{dt} D_{a_j}-D_{a_j} s_j=- D_{a_j}^*F_{a_j} \\
 &&=  - S_{ij}^{-1}\circ D_{a_i}^*F_{a_i}  \circ S_{ij}
=  S_{ij}^{-1}\circ  (\frac {d    D_{a_i}}{dt}    -   D_{a_i}  s_i) \circ S_{ij}
    \end{eqnarray*}
   with $s_j=S_{ij}^{-1}  s_i S_{ij}+S_{ij}^{-1}\circ \frac d{dt} S_{ij}$.

Similarly to \eqref{a12},  an iterative argument using equation \eqref{as} yields
\begin{eqnarray}\label{a7'}
          &&    \int_{\delta}^{t_1} \int_{\frac 12 U_i} |\nabla^{k} s_i|^2\,dv\,dt  \leq C\int_{\delta}^{t_1}\int_{U_i} |\nabla_{a}^{k}
    F_{a}|^2+\cdots +|F_{a}|^2\,dv \,dt
    \end{eqnarray}
    for any integer $k\geq 1$.

    Using \eqref{4.8}, we have
\begin{eqnarray}\label{a7'}
          &&   \int_M|\nabla^{k}\frac {da}{dt}|^2\,dv\,dt \leq C\int_M |\nabla_A^{k}
    F_A|^2+\cdots +|F_A|^2\,dv  \nonumber
    \end{eqnarray}
    for $t\in [\delta, t_1 ]$ and any integer $k\geq 2$.
This shows that $a(t)$ is a smooth solution of (\ref{4.5}) in $M\times (0,t_1]$.

Using Lemma \ref{Lemma 1}, we have
\begin{equation}\int_0^{t_1}\int_M \abs{\pfrac{A^k}{t}}^2 dv\,dt\leq  C\left(1+ \frac{t_1}{r_0^2}\right)\int_M  |F_{A^k_0}|^2
  \,dv \end{equation}
and
 \begin{eqnarray}\label{3.28}
\frac {d}{d t} \int_M  |A^k(t) |^2 dv&=&2   \int_M   \left <\frac{d }{d t}A^k(t), A^k(t)\right >
  \,dv\\
  &\leq &  \int_M \abs{\pfrac{A^k}{t}}^2 dv + \int_M  |A^k(t) |^2 dv. \nonumber \end{eqnarray}
Using Gronwall's inequality and \eqref{3.28}, we have
\begin{equation*}
\int_M |A^k(t)|^2\,dv
\leq e^{t_1}\left[
\int_M |A^k(0)|^2\,dv
+C\left(1+\frac{t_1}{r_0^2}\right)
\int_M |F_{A^k_0}|^2\,dv
\right].
\end{equation*}
   In $U_i$, we note
\begin{equation*}
 a^k_{i}(t)= (S_i^k)^{-1}(t)  d S^k_{i}(t) + (S_i^k)^{-1}(t) A^k(t)S^k_{i}(t) .\end{equation*}
  As $k\to \infty$, $S^k_{i}$ weakly converges to $S_i$ in $L^{\infty}([0, t_1]; H^{1,2}(U_i))$ and $ H^{1,2}([0, t_1]; L^2(U_i))$ up to a subsequence.
Furthermore, $A^k$ weakly converges to $A$ in $L^{\infty}(0, t_1; L^{2}(U_i))$ and $H^{1,2}([0, t_1]; L^2(U_i))$ such that $A$ is a weak solution of the Yang-Mills flow with initial value $A_0\in H^{1,2}$.
This shows that $A(t)$ is a weak solution of  the Yang-Mills flow (\ref{YMF}) with
 $A\in C^0\left ( [0, t_1]; L^2( \Omega^1(Ad \,E ))\right )\cap H^{1,2}\left ([0, t_1]; L^2(\Omega^1(Ad \,E ))\right )$ and
		 $F_A\in C^0\left ( [0, t_1]; L^2(\Omega^2(Ad \,E ))\right )$. Fix $\tilde t>0$. In each $U_i$, $\tilde A(t)=S_i^*(\tilde t)(A(x,t))$ is   a weak solution of the Yang-Mills flow with and smooth for $t>0$.

\end{proof}

We  generalize the result of Struwe \cite {St3} on existence of weak solutions of the Yang-Mills flow in four dimensions to higher dimensions by proving Theorem \ref{Theorem 1.2}.

\begin{proof}[Proof of Theorem   \ref{Theorem 1.2}] Following from Theorem \ref{Theorem 4.2}, there exists a weak solution $A(t)\in M\times [0, t_1]$ of the Yang-Mills flow with initial connection $A_0$ satisfying $F_{A_0}\in L^{n/2}(M)$ for some $t_1>0$. According to Lemmas \ref{Lemma  5}-\ref{Lemma  6},    we can start again at the time $t_1$ as a new initial time and extend the solution to the maximal existence time $T_1>0$ such that as $t_i \to  T_1$,
 there is a  constant $\varepsilon_1>0$  and at least one singular point  $x_0\in M$, which is characterized  by the condition
\[ \limsup_{t_i \to T_1}  \int_{B_R(x_0)} |F(x, t_i,)|^{n/2}\,dv \geq \varepsilon_1 \]
 for any $R\in (0, R_0]$ with some $R_0>0$. Then for any $t\in (0, T_1)$, the weak solution $A({t})$ is gauge-equivalent to a smooth solution.
\end{proof}

\section{ Monotonicity Formula and Applications}
We can now prove Theorems \ref{Theorem1.3'}-\ref{Theorem 1.4}, but first we recall a   monotonicity formula for the
Yang-Mills flow that is analogous to the one given by
Struwe \cite {St2} for the heat flow of harmonic maps.
For any $x_0\in M$, there exists a geodesic ball $B_{i(M)}(x_0)\subset M$ for
some $i(M)>0$.
Let $z=(x,t)$ be points in $M\times\R$.  Then for some fixed point
$z_0=(x_0,t_0)\in M\times \R_+$, we let
$$T_R(t_0) :=\left \{ z=(x,t) : t_0-4R^2<t<t_0-R^2, x\in M\right \}\,,$$
and
$$P_R(z_0) :=B_R(x_0)\times [t_0-R^2, t_0 ]\,,$$
where $B_R(x_0)\subset M$ denotes the geodesic ball at the center
$x_0$ with radius $R$. For simplicity, we denote
$T_R(0,0)$ and $P_R(0,0)$ by   $T_R$ and $P_R$ respectively.

The fundamental solution to the (backward) heat equation with singularity at $z_0$ is given by
$$G_{z_0}(z)=\frac 1{(4\pi (t_0-t))^{n/2}}\text{exp} \left (-\frac
{(x-x_0)^2}{4(t_0-t)}\right )\,, \quad t<t_0$$

Assume that  $A$ is a solution of the heat flow (1.3) in $M\times
\R_+$. For some smooth cut-off function $\phi$  with compact support in
$B_{i(M)}(x_0)$ such that $\phi =1$ in $B_{i(M)/2}(x_0)$, a Struwe-type monotonicity in $B_{i(M)/2}(x_0)$   was established in  \cite{CS}, \cite{Ha1}  and  \cite {HT1} in the following:
\begin{align}\label{GM1}
 &\int_{R_1}^{R_2} R\int_{T_{R}} |t-t_0|  \left | \frac {\partial A}{\partial
t}-\frac {x_i-{x_0}_i} {2|t-t_0|} g^{i\a}F^{\a\b}dx_{\b} \right
|^2  \,\phi^2\, G\,dv\,dt\,dR\\
&+\int_{T_{R_1}(z_0)} R_1^2 \, |F_A|^2\,\phi^2 \,G_{z_0}\,dv\,dt\nonumber \\
\leq &   e^{C(R_2-R_1)}
\int_{T_{R_2}(z_0)} R_2^2 \, |F_A|^2\,\phi^2 \,G_{z_0}\,dvdt
 + C(R_2 -R_1)\text
{YM}(A_0)\nonumber
\end{align} for any two $0<R_1\leq R_2<\min\{i(M),\frac 1 2 \sqrt{t_0}\}$.

As a consequence of equation \eqref{GM1},  we   prove Theorem \ref{Theorem1.3'}
\begin{proof} [Proof of Theorem    \ref{Theorem1.3'}] Our proof is to modify the one on harmonic maps by Struwe \cite{St2}.
Suppose there exists a positive $R_0$  such that
\begin{equation}\label{5.17}
   \sup_{t\leq T_1, x_0\in M} \int_{B_{R_0}(x_0)} |F_A (x, t)|^{n/2}   dv< \varepsilon_1
\end{equation} for a sufficiently small $\varepsilon_1 > 0$.
Following from Lemmas \ref{Lemma 5}-\ref{Lemma 6}, $\nabla^lF_A$ for any integer $l\geq 0$ is uniformly bounded as $t \rightarrow T_1$, contradicting the hypothesis of $T_1$. Thus, there exist a sequence $R_k \to 0$, points $x_k\in M$ and $t_k<T_1$  with $t_k \rightarrow T_1$ such that
\begin{equation*}
     \varepsilon_1-\frac 1 k\leq    \int_{B_{R_k}(x_k)} |F_A (x, t_k)|^{n/2}   dv=\max_{0\leq t\leq t_k}  \max_{ y\in M}\int_{B_{R_k}(y)} |F_A (x, t)|^{n/2}   dv < \varepsilon_1 .
\end{equation*}
 Let
\begin{equation*}
 A_k(\tilde x,\tilde t) :=R_k A^i(x_k+R_k\tilde x, t_k +R_k^2\tilde t)d\tilde x_i.
\end{equation*}
Then $A_k(\tilde x, \tilde t)$ solves the Yang-Mills flow in $B_{\frac {R_0}{R_k}}(0)\times (-\frac{t_k}{R_k^2},0]$ and satisfies
 \begin{align}\label{5.3}
 \varepsilon_1-\frac 1 k &\leq    \int_{B_{1}(0)} |F_{A_k} (\tilde x, 0)|^{n/2}   d\tilde v\\
 &   =\sup_{  -\frac{t_k}{R_k^2}\leq t\leq  0}  \sup_{\tilde x_0\in B_{R_0/R_k}(0) } \int_{B_{1}(x_0)} |F_{A_k} (\tilde v, \tilde t)|^{n/2}   d\tilde v < \varepsilon_1 .\nonumber
\end{align}
According to  Lemma  \ref{Lemma 6}, there exists a uniform constant $C_l > 0$ independent of $k$ such that
\[\|\nabla^lF_{A_k}(\cdot, \tilde t)\|_{W^{1, 2}(B_1(\tilde x_0))}\leq C_l\]  for
 $\tilde t\in [-L, 0)$, for any integer $l\geq 0$, $x_0\in\R^n$ and $L>1$. Passing to a subsequence, we may assume that as $k\to \infty$, $A_k \to \bar{A}$ in $C^{\infty}_{loc}(\R^n\times (-\infty, 0) )$ up to gauge transformation. Using \eqref{5.3} and Lemma \ref{Lemma 3},     $\bar{A}$ is a nontrivial  solution of the Yang-Mills flow in $\R^n \times (-\infty,0)$ such that
\begin{equation}\label{D5}
\sup_{  -\infty< t\leq   0}   \int_{B_{1}(0)} |F_{\bar A} (\tilde x, \tilde t)|^{n/2}   d\tilde v=\varepsilon_1.
\end{equation}
Let  $\phi$ be a  smooth cut-off function  with compact support in
$B_{i(M)}(x_0)$ such that $\phi =1$ in $B_{i(M)/2}(x_0)$. For each $L>1$,   we apply \eqref{GM1} to get
\begin{align}\label{GM2}
 &\int_{R_k}^{LR_k} R\int_{T_{R}} |t-t_k|  \left | \frac {\partial A}{\partial
t}-\frac {x_i-{x_k}_i} {2|t-t_k|} g^{i\a}F_A^{\a\b}dx_{\b} \right
|^2  \,\phi^2\, G_{z_k}\,dv\,dt\,dR\\
&+\int_{T_{R_k}(z_k)} R_k^2 \, |F_A|^2\,\phi^2 \,G_{z_k}\,dv\,dt\nonumber \\
\leq &   e^{R_k (L -1) }
\int_{T_{LR_k}(z_k)} L^2R_k^2 \, |F_A|^2\,\phi^2 \,G_{z_k}\,dvdt
 + CR_k (L -1)\text
{YM}(A_0).\nonumber
\end{align}
Then it follows from  \eqref{GM1} that
\[\lim_{k\to \infty}\int_{T_{LR_k}(z_k)} L^2R_k^2 \, |F_A|^2\,\phi^2 \,G_{z_k}\,dvdt= \lim_{k\to \infty}\int_{T_{R_k}(z_k)} R_k^2 \, |F_A|^2\,\phi^2 \,G_{z_k}\,dv\,dt.\]
Let  $\tilde \phi$ be any cut-off function in $C_0^{\infty}(\R^n)$. As $k\to\infty$, it  follows from \eqref{GM2} that
\begin{align*}
 &\quad\int^{L}_{1} \tilde R\int_{T_{\tilde R}} |\tilde t|  \left | \frac {\partial \bar A}{\partial
t}-\frac {\bar x_{\a}} {2|\tilde t|} F_{\bar A}^{\a\b}d\tilde x_{\b} \right
|^2  \, \tilde \phi^2G\,\, d\tilde v\,d\tilde t\,d\tilde R\\
&\leq \lim_{k\to \infty}\int^{L}_{1} \tilde R\int_{T_{\tilde R}} |\tilde t|  \left | \frac {\partial  A_k}{\partial
t}-\frac {\bar x_{\a}} {2|\tilde t|} F_{A_k}^{\a\b}d\tilde x_{\b} \right
|^2  \,  \tilde \phi^2 G\,\,d\tilde v\,d\tilde t\,d\tilde R\nonumber\\
 &=\lim_{k\to \infty}\int_{R_k}^{LR_k} R\int_{T_{R}} |t-t_k|  \left | \frac {\partial A}{\partial
t}-\frac {x_i-{x_k}_i} {2|t-t_k|} g^{i\a}F_A^{\a\b}dx_{\b} \right
|^2  \,\tilde \phi^2 (\frac {x-x_k}{R_k})\, G_{z_k}\,dv\,dt\,dR\nonumber\\
&=0.\nonumber
\end{align*}
Therefore,
 $\bar {A}(x,t)$ is a smooth solution of the Yang-Mills flow in $\R^n\times (-\infty , 0]$ satisfying
\begin{equation}
     \frac{\partial  \bar A  }{\partial \tilde t}-\frac{\tilde x_\alpha }{2|\tilde t|} F^{\alpha\beta}_{ \bar A }d\tilde x_{\beta}=0.
\end{equation}

\end{proof}

We recall an improved local  monotonicity formula in a small ball, which was stated in  \cite {HT1}  as follows.
\begin{thm}\label{Theorem 5.1} Let $A$ be a regular
solution of the heat flow equation (1.3). For some small $R_0>0$,
  let $\phi \in
C^{\infty}_0(B_{R_0}(x_0))$ be a cut-off function such that $\phi
\equiv 1$ on $B_{R_0/2}(x_0)$, $|\phi |\leq 1$ and $|\nabla\phi |\leq
\frac C {R_0}$ in $B_{R_0}(x_0) $.  For any $z_0=(x_0, t_0)\in M\times [0,T]$ and any $0 < R\leq R_0$, we define
$$\Phi (R;  A) := R^2 \int_{T_R(z_0)} |F_A|^2(x,t)\,\phi^2(x)\,G_{z_0}\,dv\,dt.$$
Then we have
\begin{align}\label{L5.1}
 \Phi (R_1;  A)\leq e^{C(R_2-R_1)}\Phi (R_2;  A)+ C(R_2 -R_1)\text
{YM}(A_0)+CR_0^{2-n}\int_{P_{R_0}(z_0)} |F_A|^2\,dv\,dt
\end{align}
 for $0<R_1\leq R_2\leq R_0$.
 \end{thm}
 \begin{proof} The proof   was outlined in \cite{HT1}, but without detailed proof. For completeness we provide detail here.
  At $x_0 \in M$, we choose a normal coordinate satisfying
$$|g_{ij}(x)-\delta_{ij}|\leq C|x-x_0|^2, \quad \left |\frac {\partial
g_{ij}}{\partial x_k}\right |\leq C |x-x_0|.$$
 Since the Yang-Mills flow (\ref{YMF}) is invariant under
translation $x\mapsto x-x_0$, $t\mapsto t-t_0$, we may assume that
$z_0=(x_0,t_0)= (0,0)$. For simplicity we denote
$G_{(0,0)}(x,t)$ by $G(x,t)$.

Then by rescaling we have
$$  \tilde A(\tilde
x) =RA_i(R\tilde x, R^2\tilde t)d\tilde x_i\, $$ where $x=R\tilde x$ and
$t=R^2\tilde t$. Subsequently,
$$|F_{\tilde A}|^2 (\tilde x, \tilde t)=R^4| F_A|^2(x,t)\,
.$$
Also, after the above rescaling, the quantity $\Phi$ becomes
 \begin{align}\label{L5.2}
\Phi (R; A)=\int_{T_1}| F_{\tilde A}|^2\, \phi (R\tilde
x)\,G(\tilde z)\sqrt { g (R\tilde x)}\,d\tilde z,
 \end{align}
where $d\tilde z=d\tilde v\,d\tilde t$.

Then we compute
  \begin{align}\label{L5}
&\frac d{dR} \Phi (R; A)=\int_{T_1}\frac d{dR}\left [ R^4  |F_A|^2
(R\tilde x, R^2\tilde t) \,\phi (R\tilde x)\, \sqrt {g
(R\tilde x)} \,\right ]  G(\tilde z) \,d\tilde z \\
&=R^3\int_{T_1}4 |F_A|^2 (R\tilde
x, R^2\tilde t) \,\phi^2 (R\tilde x) \,G(\tilde z)\,d\tilde z\nonumber\\
& +R^4\int_{T_1} \tilde x_k\frac{\partial }{\partial x_k}  |F_A|^2
(R\tilde x, R^2\tilde t) \,\phi^2(R\tilde x)\, G(\tilde
z)\,d\tilde v\,d\tilde t\nonumber\\
&+R^4\int_{T_1}2R\tilde t\frac{\partial }{\partial t} |F_A|^2 (R\tilde
x, R^2\tilde t)\, \phi^2(R\tilde x)\, \,G(\tilde
z)\,d\tilde v\,d\tilde t\nonumber\\
&+\int_{T_1} R^4  |F_A|^2 (R\tilde x, R^2\tilde t)\, \tilde
x_k\frac{\partial }{\partial x_k} \,(\phi^2  \sqrt {g}) (R\tilde
x)\, G(\tilde z)
\,d\tilde z \nonumber\\
&:= I_1+ I_2+I_3+I_4\,.\nonumber
 \end{align}

In local coordinates, we have
$$
|F_A|^2=\<F_A,\,F_A\> =\frac 1 4 \<F^{ij}dx_i\wedge dx_j,\,
F^{\a\b}dx_{\a}\wedge dx_{\b}\>\, .
$$
Let $\nabla_{A, x_k} F_A=\frac 12 \nabla_{A, x_k} F^{ij}
dx_i\wedge dx_j $ be the gauge-covariant derivative of $F_A$ with
respect to $\frac {\partial}{\partial x_k}$, where $\nabla_{A,
x_k} F^{ij}=\frac{\partial F^{ij}}{\partial x_k}+[A_k,
F^{ij}]-\sum_l\Gamma^l_{ik} F^{lj}-\sum_l\Gamma^l_{jk} F^{il}$. Then in
local coordinates, we have
$$D_A^*F_A=-\sum_{i,\a,\b}g^{i\a} \nabla_{A, x_i}F^{\a \b}dx_{\b}\,.$$
Since $A$ is compatible with the Riemannian structure, we obtain
$$\frac{\partial }{\partial x_k}|F_A|^2 =\frac 1 2
\<\nabla_{A,x_k}F^{ij}dx_i\wedge dx_j, F^{\a\b}dx_{\a}\wedge
dx_{\b}\>\,.$$ In local coordinates, the Bianchi identity is
$$D_AF_A=
\frac 12 \nabla_{A,x_k} F^{ij}dx_k\wedge dx_i\wedge dx_j=0\,,$$
which is equivalent to $\nabla_{A, x_k} F^{ij}=\nabla_{A,
x_i}F^{kj}-\nabla_{A, x_j} F^{ki}$. Then it follows from the
Bianchi  identity that
 \begin{align}\label{L5.4}
x_k\frac{\partial }{\partial x_k}|F_A|^2 &=\frac 1 2
x_k\<(\nabla_{A,x_i}F^{kj}-\nabla_{A,x_j}F^{ki})dx_i\wedge dx_j,\,
F^{\a\b}dx_{\a}\wedge dx_{\b}\>\\
&= \<\nabla_{A,x_i}(x_kF^{kj}) dx_i\wedge dx_j,\,
F^{\a\b}dx_{\a}\wedge
dx_{\b}\> -4|F_A|^2\nonumber\\
&\quad -\<x_kF^{lj}\Gamma_{ki}dx_i\wedge dx_j,\,
F^{\a\b}dx_{\a}\wedge dx_{\b}\>\,,\nonumber
  \end{align}
where $\nabla_{A, x_i} (x_kF^{kj}):=\frac{\partial }{\partial
x_i}(x_kF^{kj})+[A_i, x_k F^{kj}]-x_kF^{kl}\Gamma_{ji}^l$ is the
gauge-covariant derivative of $x_kF^{kj}$ with respect to $\frac
{\partial}{\partial x_i}$.

From Stokes' formula, we have
  \begin{align}
&\int_{T_{R}}\<\nabla_{A,x_i}(x_kF^{kj}) dx_i\wedge dx_j,\,
F^{\a\b}dx_{\a}\wedge dx_{\b}\>\,\phi^2\, G\,\sqrt g\,dz\\
&=2\int_{T_{R}}\<\nabla_{A,x_i}(x_kF^{kj})  dx_j,\,
g^{i\a}F^{\a\b}dx_{\b}\>\,\phi^2\, G\,\sqrt g\,dz\nonumber\\
&=-2\int_{T_{R}}\< x_kF^{kj}dx_j,\, g^{i\a}
\nabla_{A,x_i}F^{\a\b}dx_{\b} \>\,\phi^2 \,G\,\sqrt g\,dz\nonumber\\
&\quad -2\int_{T_{R}}\<x_kF^{kj}dx_j,\,
g^{i\a}F^{\a\b}dx_{\b}\>\,\phi^2\,\frac
{\partial G}{\partial x_i}\sqrt g\,dz\nonumber\\
&\quad -4\int_{T_{R}}\<x_kF^{kj}dx_j,\,
g^{i\a}F^{\a\b}dx_{\b}\>\,\phi\,\frac {\partial \phi }{\partial
x_i}\,G\,\sqrt g\,dz,\nonumber
  \end{align}
where we use the fact that
$$\frac 1{\sqrt g}\frac {\partial (g^{i\a}\sqrt g)}{\partial
x_i}F^{\a\b}=-g^{i\a}\Gamma_{l\a}^lF^{l\b}\,. $$ Using the above
identities, we have
  \begin{align}
  & \int_{T_{R}}x_k\frac{\partial }{\partial x_k}|F_A|^2\phi^2 G\sqrt g\,dz\\
&\geq \int_{T_{R}}2\(\<x_k F^{kj} dx_j,\, D_A^*F_A\>
-2|F_A|^2\)\, G\,\phi^2 \,\sqrt {g}\,dz \nonumber\\
&\quad -\int_{T_{R}}2\< x_kF^{kj}  dx_j,\, g^{i\a}F^{\a\b} dx_{\b}
\> \frac
{\partial G}{\partial x_i}\, \phi^2 \,\sqrt {g}\,dz\nonumber\\
&\quad  -C\int_{T_{R}} |F|^2(|x|^2 |\phi|^2+|x||\nabla\phi
|)\,G\,\sqrt {g}\,dz,\nonumber
   \end{align}
where we have used the fact that $|\Gamma_{ij}^k|\leq C|x|$.

Rescaling back to coordinates $(x,t)$, we have
  \begin{align}\label{L5.7}
I_1+I_2&=R\int_{T_R}[4 |F_A|^2+ x_k\frac {\partial}{\partial
x_k}  |F_A|^2 ](z)\,\phi^2 (x)\,G(z)\,\sqrt {g (x)}\,dz\\
&\geq 2R\int_{T_{R}} \<x_kF^{kj}dx_j, D_A^*
F_A\>\,\phi^2\,
G\,\sqrt g\,dz\nonumber\\
&\quad -2R\int_{T_{R}} \<x_kF^{kj}dx_j,\, g^{i\a}F^{\a\b}\>\frac
{\partial G}{\partial x_i}\,\phi^2\,\sqrt g\,dz\nonumber\\
&\quad -CR\int_{T_{R_1}} (|x|^2 \phi^2 +|x| |\nabla \phi
|)\,e(A)\,\phi^2\,G\,\sqrt g\,dz\,.\nonumber
   \end{align}

Applying \eqref{L5}  and  Stokes' formula, we deduce
  \begin{align}\label{L5.8}
  I_3&=2R\int_{T_{R}}t\frac {\partial}{\partial t}
|F_A|^2\,\phi^2\,G\,\sqrt g\, dv\,dt\\
&=4R\int_{T_R}t\<F_A,
\,D_A(\frac {\partial A}{\partial t})\> \,\phi^2\,G\,\sqrt g dv\,dt\nonumber\\
&=-4R\int_{T_{R}} t|\frac {\partial A}{\partial t}|^2
\,\phi^2\,G\,\sqrt g\,
dv\,dt\nonumber\\
&\quad-4R\int_{T_{R}}t \<\frac {\partial A}{\partial t},\,
g^{i\a}F^{\a\b}dx_{\b}\> \(\frac {\partial  G}{\partial
x_i}\phi^2+2\phi \frac {\partial \phi }{\partial x_i}G\)\,\sqrt
g\, dv\,dt.\nonumber
\end{align}

Note that
$$\frac {\partial G}{\partial x_i}=\frac {x_i}{2t}G.$$

Then from \eqref{L5}-\eqref{L5.8}, we get
  \begin{align}\label{L5.9}
& \frac d{dR}\Phi (R;A)  =I_1+I_2+I_3+I_4\\
&\geq  R\int_{T_{R}} |t|  \left | \frac {\partial A}{\partial
t}-\frac {x_i} {2|t|} g^{i\a}F^{\a\b}dv_{\b} \right
|^2  \,\phi^2\, G\,\sqrt g\,dv\,dt\nonumber\\
&\quad -C\int_{T_{R}}R |F_A|^2
(|x|^2\phi^2+ |x| |\nabla  \phi |) \,G\,\sqrt g\,d v\,d t\nonumber\\
&\quad -C\int_{T_{R}}R t^{-1} |F_A|^2\,\phi^2\,|
x|^4\,G\,\sqrt g \,d v\,d t,\nonumber
\end{align}
where $C > 0$ is a constant depending on the geometry of $M$. Note
 that
$$R^{-1} |x|^2G\phi^2\leq C(1+G)\phi^2, \quad R^{-1} |t|^{-1}|x|^4G\phi^2\leq
C(1+ G)\phi^2\quad \text {on }T_{R}.$$ Furthermore,  we have $\nabla \phi =0$
for $x\in B_{\frac {R_0}2 }$ and
$$G \leq c\frac
1{|t|^{n/2}}e^{-\frac {R_0^2}{16 t}}\quad \text { in
} B_{R_0}\backslash B_{\frac {R_0}2 }$$
implying
$$|x| |\nabla \phi |  G\leq \frac {C} {R^n}e^{-\frac {R_0^2}{R^2}}\leq \frac {C} {R_0^n}\frac {R_0^n} {R^n}e^{-\frac {R_0^2}{R^2}}\leq \frac {C} {R_0^n}\ \quad\mbox{in } T_R$$ for any $R\leq R_0$.

Combining these with Lemma \ref{Lemma 1}, we obtain
 \begin{align}
\frac d{dR}\Phi (R; A)\geq& -C \Phi (R; u, A)-CR\text {\text {YM}}
(A_0)\\
&-C\frac {R} {R_0^n} \int_{P_{R_0}(z_0)} |F_A|^2 dv\,dt.\nonumber
 \end{align}
The claim then follows from integrating the above inequality with respect to $R$.
 \end{proof}

We now recall the well-known Bochner inequality.

\begin{lem} \label{L5.3}Let $A$ be a regular solution of the Yang-Mills flow (\ref{YMF}) in $M\times [0,T]$.  Then
 \begin{align}
 &\(\frac{\partial }{\partial t} - \Delta_M \)
       |F_A|^2   \leq
    C(|F_A| + |R_M|)  |F_A|^2,
 \end{align}
where $C > 0$ is a constant independent of the connection $A$ and  $\Delta_M=\frac 1 {\sqrt
g}\frac {\partial}{\partial x_j} (g^{ij}\frac {\partial }{\partial
x_i})$ is the Laplace-Beltrami operator on $M$
 \end{lem}

As application of the  local  monotonicity, we have an improved $\varepsilon$-regularity theorem as in \cite{HT1}:

 \begin {thm}\label{Theorem 6} Suppose that $A$ is a solution of the heat
flow equation (1.3) in $M\times [0, \infty )$ with initial connection
$A_0$. If there exists a constant $0 < \varepsilon_0<i(M)$ such
that
$$R^{2-n}\int_{P_R(x_0, t_0)}|F_A|^2\,dv\,dt
    \leq \varepsilon_0$$ for  some $0<R<\min \{\varepsilon_0, \frac
{t_0^{1/2}}2\}$,
  then we have
$$\sup_{P_{\delta R}(x_0,t_0)} |F_A|^2 \leq C(\delta R)^{-4}\,,$$
where $C > 0$ is a constant depending on $M$, $\delta >0$, $\text
{YM}(A_0)$ and $R$.
\end{thm}
  \begin{proof} This result was outlined in \cite{HT1}. Here, we would like to give a more detailed proof.
 Without loss of generality, assume that
$(x_0,t_0)=(0,0)$, $r_1=\delta R$, $\delta\in ]0,1/2]$. As in
\cite{Schoen}  and  \cite {St2}, we choose $r_0<r_1$ such that
 \begin{align} \label{5.18}
(r_1-r_0)^4\sup_{{\bar P}_{r_0}}|F_A|^2  =
   \max_{0\leq r\leq r_1}\left [ (r_1-r)^4\sup_{P_r} |F_A|^2\right ]  \end{align}
   and find $(x_1,t_1)\in P_{r_0}$ such that
$$e_0:=|F_A|^2 (x_1,t_1)=\sup_{{\bar P}_{r_0}} |F_A|^2 \,.$$
 We claim that
$$e_0\leq 16 (r_1-r_0)^{-4}\,.$$
Otherwise, we have
$$\rho_0=e_0^{-1/4}\leq \frac{r_1-r_0}2\,.$$
By rescaling we have
$$B(\tilde x)=\rho_0\,A(x_0+\rho_0\tilde x, t_0+\rho_0^2\tilde t)\,.$$
Then
$$|F_B|^2=\rho_0^4\,|F_A|^2.$$
Setting
$$
e_{\rho_0}(B)=| F_B|^2=\rho_0^{4}\,|F_A|^2,$$ we thus have
\begin{align} \label{5.19}
1=e_{\rho_0}(0,0)&\leq \sup_{\bar P_1}e_{\rho_0}(B)
  = \rho_0^4\sup_{P_{\rho_0 (x_0,t_0)}}   |F_A|^2  \\
& \leq \rho_0^4\sup_{P_{\frac{r_1+r_0}2}}  |F_A|^2  \nonumber\\
& \leq \rho_0^4\(\frac{r_1 - r_0}{2}\)^{-4}\(r_1 - r_0\)^4e_0
     = 16\,,\nonumber\end{align}
     where the last step derives from \eqref{5.18}.
This implies that
$$|F_B|^2\leq 16\quad \mbox{on } \bar P_1\,.$$
Let $\tilde M :=\{ \tilde x: x_1+\rho_0\tilde x\in M\}$
  and $\tilde P_1 :=\{ (\tilde x,\tilde t): \quad (\tilde x,\tilde
t)\in B_1(0)\times [-1,1]\}$.

Following from Lemma \ref{L5.3}, we have
\begin{align} \label{?}
(\frac {\partial}{\partial \tilde  t} - \Delta_{\tilde M} )
 |F_B|^2 &=\rho_0^6 \,(\frac {\partial}{\partial  t} -
\Delta_{M} ) |F_A|^2\\
   &\leq C (|F_B|+\rho_0^2)|F_B|^2 \leq C e_{\rho_0},
\text { in } \tilde P_1\,,\nonumber \end{align}
where the constant $C > 0$
depends on $i(M)$ and  $\sup_{x\in M} |R_M|$.
Then Moser's parabolic Harnack inequality yields
\begin{align} \label{5.21}
1=e_{\rho_0}(0,0)&\leq C\int_{\tilde P_1}e_{\rho_0}(B)\,d\tilde v\,d\tilde t
\leq C\rho_0^{-n}\int_{P_{\rho_0}} \rho_0^2  |F_A|^2 \,dv\,dt\,.
\end{align}
Let $\phi \in C^{\infty}_0(B_{\frac 34}R(x_0))$ be a cut-off
function such that $\phi \equiv 1$ on $B_{R/2}$, $|\phi |\leq 1$
and $|\nabla\phi |\leq \frac C R$ in $B_{\frac 34 R}\backslash
B_{R/2}$.

For $r, \sigma\in (0, r_1)$ and any $z_0=(x_0,t_0)\in P_{r_1}$ and taking $R_1=\sigma$,  $R_2=\frac 12 R$ in Theorem \ref{Theorem 5.1}, we
obtain
\begin{align} \label{5.22}
&\sigma^{2-n}\int_{P_{\sigma}(z_0)} |F_A|^2\,dv\,dt \leq C\int_{T_{\sigma }(t_0+2\sigma^2)} \sigma^2 |F_A|^2\,
G_{(x_0, t_0+2\sigma
^2)}\,\phi^2\,dv\,dt \\
&\leq C\int_{T_{\frac  12 R}(t_0+2\sigma^2)}R^2 |F_A|^2\, G_{(x_0, t_0+2\sigma^2)}\,\phi^2\,dv\,dt\nonumber\\
&\quad +CR\, \text {YM} (A_0) +CR^{2-n}\int_{P_{\frac 34R}(x_0, t_0+2\sigma^2 )}|F_A|^2\,dv\,dt\nonumber\\
&\leq C\int_{P_R(z_0)} R^{2-n} |F_A|^2\,dv\,dt +CR\, \text
{\text {YM}} (A_0)\leq C\varepsilon_0 \,,\nonumber
\end{align}
where we note there exists a constant $C > 0$  for $t_0+2\sigma^2 -R^2\leq t\leq
t_0+2\sigma^2-\frac {R^2}2$, such that
$$G_{(x_0,t_0+2\sigma^2)} =  \frac 1{(4\pi (t_0+2\sigma^2-t))^{n/2}}\text{ exp} \left (-\frac
{(x-x_0)^2}{4(t_0+2\sigma^2-t)}\right ) \leq CR^{-n}.
$$

 Taking $\sigma =\rho_0$ in
 in \eqref{5.22} and letting $\varepsilon_0 > 0$  be
sufficiently small, we obtain a contradiction of \eqref{5.21}. Therefore, we have
$$e_0\leq \frac {16}{(r_1-r_0)^4},$$
which implies
$$\sup_{P_{r_1/2}}|F_A|^2
    \leq 16 \,(r_1)^{-4}(r_1-r_0)^4e_0\leq 256\, (r_1)^{-4}\,.$$
This proves Theorem  \ref{Theorem 6}.
\end{proof}

We can now provide a proof of Theorem  \ref{Theorem 1.3}.

\begin{proof}[Proof of Theorem  \ref{Theorem 1.3}]

  Let  $T_1$ be the maximal existence time and finite in Theorem \ref{Theorem 1.2}  satisfying \eqref{Cri} and $T_1\geq C i(M)^2$. Choose $R_0= \min\{i(M), \varepsilon_0\}$, then according to Lemma \ref{Lemma 1},  H\"older's inequality and choosing a sufficiently small $\varepsilon_2 > 0$,  we have
 $$R_0^{4-n}\int_{B_{R_0}(x_0)}|F_{A(t)}|^2\,dv\leq  R_0^{4-n}\int_{M}|F_{A_0}|^2\,dv\leq C \varepsilon_0^{4-n}\varepsilon_2^{4/n}<\varepsilon_0 $$ for  any $t<T_1$.  Following from Theorem  \ref{Theorem 6} and Lemma \ref{Lemma 6}, it is apparent that $T_1$ is not the maximal existence time. This concludes the proof of Theorem  \ref{Theorem 1.3}.
\end{proof}

Finally,  we improve a result of Hong-Tian \cite{HT1}, by proving Theorem  \ref{Theorem 1.4}.

\begin{proof}[Proof of Theorem   \ref{Theorem 1.4}]
Let $t_k\to \infty$ be any sequence.
From Lemma \ref{Lemma 1},  we know that $\int_0^{\infty }\int_M
 |\partial_tA|^2  \,dvdt$ is finite. Then for any  sequence  of $t_k\to\infty $, we have
 $$\int_{t_k-1}^{t_k}\int_M \left|\frac {\partial A}{\partial t}\right|^2\,dv dt\to 0.
 $$
  According to the mean value formula, there exists $\tilde t_k\in [t_k-1, t_k]$ such that
  \begin{align} \label{tilde}\int_M \left|\frac {\partial A}{\partial t}\right|^2(\cdot, \tilde t_k)\,dv \to 0.
 \end{align}
For the sequence of $\{A(t_k)\}$, we  define
$$\Sigma \{A(t_k)\} :=\bigcap_{\varepsilon_0>r>0} \left\{x\in M:\quad
\liminf_{k\to\infty}\, r^{4-n}\int_{B_r(x)}   |F_A|^2  (\cdot
,t_k)|^2\,dv\geq \varepsilon_3\right \}\,,$$ where
$\varepsilon_3 > 0$ is determined in the sequel.

 Applying H\"older's inequality, we obtain
  \begin{align}
  &\int_M | A(t_k)-A(\tilde t_k)|^2\,dv =\int_M \left |\int_{\tilde t_k}^{t_k} \frac {\partial A}{\partial t}\,dt \right |^2\,dv\\
  &\leq C\left (\int_{t_k-1}^{t_k}\int_M \left|\frac {\partial A}{\partial t}\right|^2\,dv dt\right )\to 0.\nonumber
 \end{align}
 which implies that $A(t_k)$ and $A(\tilde t_k)$ have the same limit $A_{\infty}$ in $L^2(M)$.

Let $\phi$ be a cut-off function in $M$. Then it follows from a local energy identity that
 \begin{align} \label{Loc1}
\left |\int_M (|F_A (\cdot ,s)|^2-|F_A(\cdot , \tau ) |^2 )\phi^2 \,dv\right |\leq
C\int_{t_k-1}^{t_k}\int_M |\partial_t
A|^2+ |\partial_t A| |F_A||\nabla \phi |\,dv\,dt.
 \end{align}
for two $s,\tau\in [t_k-1,t_k]$. Subsequently, we have
 $\Sigma \{A(t_k)\}= \Sigma \{A(\tilde t_k)\} =:\Sigma$.

Then, let $\mu$ be
the limit Radon measure of $\mu_k=|F_{A(t_k)}|^2dv$. Since $\int_M|F_{A(t)}|^2dv$ is non-decreasing in $t$, it follows from
\eqref{Loc1} that for
any continuous function $\phi$ with compact support in $X$, we have
$$\mu (\phi)=\lim_{k\to\infty} \mu_k (\phi)=\lim_{k\to\infty}\int_M \phi  |F_{A(t_k)}|^2\,dv=\lim_{k\to\infty}\int_M \phi  |F_{A(\tilde t_k)}|^2\,dv .
$$
 Following from \cite{T} and using a lemma of \cite{HT1},  $e^{ar^2} r^{4-n} \mu (B(x))$ for $x\in M$  is
nondecreasing in $r$ according to \eqref{tilde}. Then the density $$\Theta (\mu , x)=\lim_{r\to 0}
r^{4-n} \mu (B_r(x))
$$
exists for every $x\in X$ and $x\in  \Sigma
(\{A_{t_k}\}_{k=1}^{\infty})$ if and only if $\Theta (\mu , x)\geq
\varepsilon_4>0$.

Now, suppose $x_1\in M\backslash\Sigma $, then there exists a $r_1$
such that, for sufficiently large $t_k$, we have
$$
r_1^{4-n}\int_{B_{r_1}(x_1)}|F_{A(t_k)}|^2\,dv<\varepsilon_4\,.$$

Letting $s=t_{k}-r_1^2$ and $\tau =t_k$ in
\eqref{Loc1}, we have
 	\begin{align*}
&r_1^{2-n}\int_{P_{\frac 12 r_1}(x_1, t_k)} |F_{A}|^2
\,dvdt\\
&\leq C r_1^{4-n}\int_{B_{r_1}}  |F_{A}(\cdot ,t_k )|^2 \,dv
 +C\,r_1^{2-n} \int_{t_k-1}^{t_k} \int_M |\partial_t A|^2 \,dv\,dt\\
&\quad +C\,r_1^{2-n} \text {YMH} (A_0, u_0)^{1/2}\left ( \int_{t_k-1}^{t_k} \int_M |\partial_t A|^2 \,dv\,dt\right )^{1/2} \\
&\leq 2C \varepsilon_4\leq \varepsilon_0\,
 	\end{align*}
for a sufficiently small $\varepsilon_4 > 0$, where $\varepsilon_0 > 0$ is the constant provided by Theorem \ref{Theorem 6}.
According to Theorem  \ref{Theorem 6}, we have
$$ |F_A(x ,t_k )|^2  \leq C( \delta r_1)^{-4}\,$$
uniformly on parabolic cylinders $P_{\frac 12\delta r_1}(x_1, t_k)$.

 For $x_0\in M\backslash \Sigma$, there
exists a small $R_0$ such that
$$|F_{A}(x ,t_k) |^2\leq C\quad \text {in } P_{R_0}(x_0, t_{k})$$
for sufficiently large $t_{k}$.
By induction, we assume that
$$|\nabla^j_A F_A (x ,t_k)|^2 \leq C \quad  \mbox {in } P_{R_j}(x_0, t_k)$$
for $j=0,..., l-1$. Using a similar proof to that of Lemma \ref{L5.3},
we deduce that
 	\begin{align} \label{H-Mo}
(\frac {\partial }{\partial t}-\Delta_M)  |\nabla_A^j F_A|^2
+2 |\nabla_A^{j+1} F_A|^2 ) \leq C
 (|\nabla_A^j F_A|^2+1)
 	\end{align}
in $P_{R_j}(x_0, t_k)$ with  small $R_j$ and $j=1,2, ... ,l$.
From \eqref{H-Mo} with $j=1,... ,l-1$, we obtain
$$\int_{P_{R_j}} |\nabla_A^j F_A(x ,t_k) |^2dv\leq C$$
for $j=0,..., l$. Using \eqref{H-Mo} with $j=l$, Moser's parabolic
estimate yields
$$|\nabla_{A}^lF_A(x ,t_k)|^2 \leq  C\quad \text{in }P_{\delta R}\,,$$
for some $\delta >0$.

Following from Uhlenbeck's Theorem in \cite {U2}, there exists
a subsequence $\{k'\}$ of  $\{k\}$ and gauge transformations
$\sigma (k')$, such that $\sigma (k')A(t_{k'})$ converges to a  Yang-Mills connection $A_{\infty}$ weakly in the
$H^{1,2}$-topology on any compact subset outside $\Sigma$ such that
$A$ is a solution of the Yang-Mills equation
outside $\Sigma$. Furthermore, according to a result in \cite {DK}, we
know that for every integer $l\geq 1$, $\|\nabla_{A}^l (\sigma (k')
A(t_{k'}))\|_{W^{l,2} B_{\frac 12 \delta R} (x_0)}$     are
also uniformly bounded. Subsequently,  $\sigma (k')A(t_{k'})$ converges to $A_{\infty}$ in $C^{\infty}_{loc}
(M\backslash \Sigma )$.
Then, for any $x_1\in
M\backslash\Sigma$, there exists a $r_2$ such that $B_{
r_2}(x_1)\subset M\backslash \Sigma$, which implies that the singular
set $\Sigma$ is closed in $M$.

We claim that $\mathcal H^{n-4}(\Sigma )\leq CYM(A_0$). For a given
$\delta >0$, we have a finite collection  $J$ of balls
$\{B_{r_i}(x_i)\}$, $r_i\leq \delta $, such that
$\{B_{r_i}(x_i)\}_{i\in J} $  is a cover of $\Sigma$ and
$B_{r_i/2}(x_i)\cap B_{r_j/2}(x_j)=\emptyset $ for $i\neq j$ with
$i,j\in J$ due to the fact that $\Sigma$ is closed.
Note that $e^{ar^2} r^{4-n} \mu (B_r(x_i))$ is
nondecreasing in $r$. Then  for $x_i\in\Sigma$ with $i\in J$,   we have
\begin{align*}
\varepsilon_1&\leq r_i^{4-n} \mu (B_r(x_i))=\lim_{k\to\infty} r_i^{4-n}\int_{B_{r_i}(x_i) } |F_A(\tilde t_k)|^2\,dv\\
&\leq C\lim_{k\to\infty} \delta^{4-n} \int_{B_{\delta}(x_i)}  |F_A(\tilde t_k)|^2\,dv.
\end{align*}
Then for
sufficiently large $t_k$ and summing on $i\in J$, we  get
\begin{align*}\sum_i r_i^{n-4}\leq \frac C {\varepsilon_1} \sum_i \frac { r_i^{n-4}} {\delta^{n-4}}\int_{B_{\delta}(x_i)}  |F_A(\tilde t_k)|^2\,dv \leq C_1\,\mbox{YM} (A_0; M)\,.\end{align*}
 It follows that
\begin{align*}
\mathcal H^{n-4}(\Sigma )\leq  C_1\,\mbox{YM} (A_0; M).	\end{align*}

 Applying Fatou's lemma, there exists a nonnegative Radon measure $\nu$ on
$M$ such that
$$\mu =|F_{A_{\infty}}|^2dv +\nu .
$$
Then it follows from \cite {T} that $\nu(x)=\Theta (\mu, x)
H^{n-4}\llcorner \Sigma $ for
$H^{n-4}$-a.e. $x\in\Sigma$, and according to \cite{T} and \cite{HT1} it follows that $\Sigma$ is $(n-4)$-rectifiable, as required.
\end{proof}

\medskip

\noindent{\bf Data availability.} No data  were generated or analyzed during the current study.
No conflict of interest for this paper.

\end{document}